\documentclass{amsart}
\usepackage{enumitem, subcaption, color, mathtools, url, mathrsfs, amsmath,amssymb,amsthm,color,mathrsfs, geometry, nicefrac, dutchcal, array}

\usepackage{physics}
\usepackage[T1]{fontenc}
\usepackage{parskip}
\usepackage{tcolorbox}
\usepackage{bbm}
\usepackage[linewidth=1pt]{mdframed}
\usepackage{ulem, censor}
\numberwithin{equation}{section}

\newgeometry{vmargin={32mm}, hmargin={30mm,30mm}}

\theoremstyle{plain}
\newtheorem{theorem}{Theorem}
\newtheorem{lemma}[theorem]{Lemma}

\theoremstyle{definition}
\newtheorem{definition}[theorem]{Definition}

\theoremstyle{definition}
\newtheorem{remark}[theorem]{Remark}

\newcommand{\nc}{\newcommand}
\nc{\R}{\mathbb{R}}
\nc{\C}{\mathbb{C}}
\nc{\mrm}{\mathrm}
\nc{\mL}{\mrm{L}}
\nc{\mF}{\mrm{F}}
\nc{\mC}{\mrm{C}}
\nc{\mH}{\mrm{H}}
\nc{\mW}{\mrm{W}}
\nc{\mV}{\mrm{V}}
\nc{\mM}{\mrm{M}}
\nc{\mK}{\mrm{K}}
\nc{\mD}{\mrm{D}}
\nc{\mB}{\mrm{B}}
\nc{\mR}{\mrm{R}}
\nc{\mX}{\mrm{X}}
\nc{\mY}{\mrm{Y}}
\nc{\mS}{\mrm{S}}

\nc{\Ec}{\mrm{E_c}}
\nc{\calL}{\mathcal{L}}
\nc{\loc}{\mrm{loc}}
\nc{\comp}{c}
\nc{\supp}{\mrm{supp}}
\nc{\Hardy}{\mathfrak{H}}
\nc{\calH}{\mathcal{H}}
\nc{\ctru}{\mathfrak{u}}
\nc{\ctrv}{\mathfrak{v}}
\nc{\bc}{\boldsymbol{c}}
\nc{\be}{\boldsymbol{e}}
\nc{\br}{\boldsymbol{r}}
\nc{\bs}{\boldsymbol{s}}
\nc{\bt}{\boldsymbol{t}}
\nc{\bw}{\boldsymbol{w}}
\nc{\bx}{\boldsymbol{x}}
\nc{\by}{\boldsymbol{y}}
\nc{\bz}{\boldsymbol{z}}
\nc{\lbr}{\lbrack}
\nc{\rbr}{\rbrack}
\nc{\dsp}{\displaystyle}
\nc{\vphi}{\varphi}

\newcommand{\lsim}{\raisebox{-0.13cm}{~\shortstack{$<$ \\[-0.07cm]
      $\sim$}}~}

\begin{document}
	\title[Galerkin approximations and canonical best-approximations]{On the relation between Galerkin approximations and canonical best-approximations of solutions to some non-linear Schrödinger equations}
	
	\author{Muhammad Hassan}\address{Institute of Mathematics, Ecole Polytechnique Fédérale de Lausanne (EPFL), 1015 Lausanne, Switzerland.}
	\author{Yvon Maday}\address{Sorbonne Université, CNRS, Université Paris Cité, Laboratoire Jacques-Louis Lions (LJLL), F-75005 Paris, France and Institut Universitaire de France,Paris, France.}
	\author{Yipeng Wang} \address{Sorbonne Université, CNRS, Université Paris Cité, Laboratoire Jacques-Louis Lions (LJLL), F-75005 Paris, France.}

	\begin{abstract} 
		In this paper, we establish a superconvergence property of Galerkin approximations to some non-linear Schrödinger equations of Gross-Pitaevskii type. More precisely, denoting by $u^*\in X \subseteq H^1(\Omega)$ the exact solution to such an equation, by $\{X_{\delta}\}_{\delta >0}$, a sequence of conforming subspaces of $X$ satisfying the approximation property, by $u_\delta^*\in X_{\delta}$ the Galerkin solution to the equation, and by $\Pi^X_{\delta} u^*$, the $(\cdot, \cdot)_{X}$-best approximation in $X_\delta$ of $u^*$, we show -- under some assumptions -- that $u_\delta^*$ converges at a higher rate to $\Pi^X_{\delta} u^*$ than to $u^*$ in both the $L^2$ norm and the canonical $H^1$ norm. Our results apply to conforming finite element discretisations as well as spectral Galerkin methods based on polynomials or Fourier (plane-wave) expansions.
	\end{abstract}
	\maketitle
	
\section{Background and Setting}\label{sec:1}

The present contribution is concerned with solutions to the following Gross-Pitaevskii type energy minimisation problems:
 \begin{align}
 \intertext{\bf{The Source Problem}}
\label{eq:1.1}E_{\rm src}^* :=& \underset{u \in X}{\text{min}} \left\{\mathcal{E}_{\rm src}(u):=\frac{1}{2}\int_{\Omega} \vert \nabla u(\bold{x})\vert^2\; d \bold{x} + \frac{1}{2}\int_{\Omega} V(\bold{x})\vert u(\bold{x})\vert^2\; d \bold{x} + \frac{1}{4}\int_{\Omega} \vert u(\bold{x})\vert^4\; d\bold{x}, - \big\langle f, u \big\rangle_{X^* \times X}\right\},\\[0.0em]
\intertext{\bf{The Eigenvalue Problem}}
E_{\rm eig}^* :=& \underset{\substack{u \in X\\ \Vert u \Vert_{L^2(\Omega)}=1}}{\text{min}} \left\{\mathcal{E}_{\rm eig}(u):=\frac{1}{2}\int_{\Omega} \vert \nabla u(\bold{x})\vert^2\; d \bold{x} + \frac{1}{2}\int_{\Omega} V(\bold{x})\vert u(\bold{x})\vert^2\; d \bold{x} + \frac{1}{4}\int_{\Omega} \vert u(\bold{x})\vert^4\; d\bold{x}  \right\}, \label{eq:1.2}
\end{align}
\vspace{2mm}

\noindent under one of the following two settings:\vspace{2mm}

\begin{description}
    \item[Setting One] $\Omega= (-1, 1)^d, ~ d \in \{1, 2, 3\}$ and  $X= H^1_0(\Omega)$. Moreover, the source function $f \in H^s(\Omega)$ for $s \geq 0$ while the effective potential $V$ is positive and $V \in H^{r_v}(\Omega)$ for $r_v > d/2$; \vspace{1mm}
    \item[Setting Two] $\Omega$ is a $d$-dimensional torus (isomorphic to the unit cell $\tilde\Omega$ of some periodic lattice $\mathbb{L}$ in $\mathbb{R}^d$), $ d \in \{1, 2, 3\}$ and $X= H^1(\Omega)$ (isomorphic to $H^1_{\rm per}(\tilde\Omega)$). Moreover, the source function $f \in H^s(\Omega)$ for $s \geq 0$ while the effective potential $V$ is bounded below by $V_{\rm min}>0$ and $V \in H^{r_v}(\Omega)$ for $r_v > d/2$.
\end{description}
\vspace{1mm}

Under these hypotheses, it is well-known that there exists a unique minimiser $u^*_{\rm src}$ to the unconstrained minimisation problem \eqref{eq:1.1} and there exists a unique positive minimiser $u^*_{\rm eig}$ to the constrained minimisation problem \eqref{eq:1.2}. Additionally, elliptic regularity theory together with a bootstrapping argument implies that
\begin{align}\label{eq:elliptic_regularity}
    u^*_{\rm src} \in H^{t_{\rm src}+2}(\Omega) \qquad \text{and} \qquad u^*_{\rm eig} \in H^{t_{\rm eig}+2}(\Omega),
\end{align}
with $t_{\rm src}= \min \{s, r_v, 3-\epsilon\}$ and $t_{\rm eig}=\min \{r_v, 3-\epsilon\}$ for any $\epsilon >0$ in the case of \textbf{Setting One} and $t_{\rm src}= \min \{s, r_v\}$ and $t_{\rm eig}=r_v$ in the case of \textbf{Setting Two}. We refer to, e.g., \cite[Chapter 3]{mclean2000strongly} for precise definitions of these higher-order Sobolev spaces.


In order to approximate these minimisers numerically, one typically introduces a sequence of finite-dimensional approximation spaces $\{X_\delta\}_{\delta > 0}~$,  $X_\delta \subset X$,  such that the standard approximation property is satisfied:
\begin{align}\label{eq:approximation_property}
 \forall u \in X\colon \qquad  \lim_{\delta \to 0^+} \; \underset{v_\delta \in X_\delta}{\inf} \Vert u - v_\delta\Vert_X =0. 
\end{align}
We then seek, for a given approximation space $X_\delta$, the solutions to the following discrete energy minimisation problems:
\begin{align}\label{eq:1.5}
E_{\rm src}^\delta :=& \underset{u_\delta \in X_\delta}{\text{min}} \left\{\mathcal{E}_{\rm src}(u_\delta):=\frac{1}{2}\int_{\Omega} \vert \nabla u_\delta(\bold{x})\vert^2\; d \bold{x} + \frac{1}{2}\int_{\Omega} V(\bold{x})\vert u_\delta(\bold{x})\vert^2\; d \bold{x} + \frac{1}{4}\int_{\Omega} \vert u_\delta(\bold{x})\vert^4\; d\bold{x} - \big\langle f, u_\delta \big\rangle_{X^* \times X}\right\},\\
E_{\rm eig}^\delta :=& \underset{\substack{u_\delta \in X_\delta\\ \Vert u_\delta \Vert_{L^2(\Omega)}=1}}{\text{min}} \left\{\mathcal{E}_{\rm eig}(u_\delta):=\frac{1}{2}\int_{\Omega} \vert \nabla u_\delta(\bold{x})\vert^2\; d \bold{x} + \frac{1}{2}\int_{\Omega} V(\bold{x})\vert u_\delta(\bold{x})\vert^2\; d \bold{x} + \frac{1}{4}\int_{\Omega} \vert u_\delta(\bold{x})\vert^4\; d\bold{x} \right\}. \label{eq:1.6}
\end{align}

The exact choice of the approximation spaces $\{X_{\delta}\}_{\delta >0}$ depends on the problem setting. Finite element or spectral polynomial discretisations may be chosen for \textbf{Setting One} in which case $\delta$ represents the maximal finite element diameter or the reciprocal of the maximal polynomial degree. In the case of \textbf{Setting Two}, a natural choice is a spectral Fourier discretisation in which case $\delta$ represents the reciprocal of the maximal wave-number (classically denoted by $N$). In either case, it is well known that for any given $\delta>0$, there exists a unique solution $u_{\delta, \rm src}^{*} \in X_{\delta}$ to the unconstrained discrete minimisation problem \eqref{eq:1.5}, and for all $\delta$ small enough, we have the error estimate
\begin{align}\label{eq:rates_1}
\Vert u^*_{\rm src} - u_{\delta, \rm src}^{*} \Vert_{L^2(\Omega)}\lsim \delta \Vert u^*_{\rm src} - u_{\delta, \rm src}^{*} \Vert_{X} \lsim \delta \underset{v_{\delta}\in X_{\delta}}{\inf}\Vert u^*_{\rm src} - v_{\delta}\Vert_{X} \lsim \delta^2 \Vert u^*_{\rm src}\Vert_{H^2(\Omega)}.
\end{align}
This estimate can be improved if the exact minimiser $ u^*_{\rm src}$ possesses additional Sobolev regularity (in the sense of Estimate \eqref{eq:elliptic_regularity}) and if the approximation spaces $\{X_{\delta}\}_{\delta >0}$ satisfy certain properties.

Under \textbf{Settings One} and \textbf{Two} and with the same choice of discretisation spaces, it has been proven\footnote{The spectral polynomial discretisation is not analysed in the cited reference but the arguments are similar.} in \cite{cances2010numerical} that there exists at least one solution $u_{\delta, \rm eig}^{*} \in X_{\delta}$ to the constrained discrete minimisation problem  \eqref{eq:1.6}. Additionally, for all $\delta$ small enough and any discrete minimiser $u_{\delta, \rm eig}^{*}$ that satisfies $\big(u_{\delta, \rm eig}^{*}, u_{\rm eig}^*\big)_{L^2(\Omega)} \geq 0$, we again have the error estimate
\begin{align}\label{eq:rates_2}
\Vert u^*_{\rm eig} - u_{\delta, \rm eig}^{*} \Vert_{L^2(\Omega)}+      \delta \Vert u^*_{\rm eig} - u_{\delta, \rm eig}^{*} \Vert_{X} \lsim \delta \underset{v_{\delta}\in X_{\delta}}{\inf}\Vert u^*_{\rm eig} - v_{\delta}\Vert_{X} \lsim \delta^2 \Vert u^*_{\rm eig}\Vert_{H^2(\Omega)}.
\end{align}
As before, this estimate can be improved if the exact constrained minimiser $ u^*_{\rm eig}$ possesses additional Sobolev regularity (in the sense of Estimate \eqref{eq:elliptic_regularity}) and if the approximation spaces $\{X_{\delta}\}_{\delta >0}$ satisfy certain properties. Note that in Inequalities \eqref{eq:rates_1} and \eqref{eq:rates_2}, and in the sequel, the notation $\lsim$ should be understood in the following sense: Let $\Lambda \subset \mathbb{R}$ denote the set of all possible discretisation parameters $\delta>0$, and let $g, h \colon \Lambda \rightarrow \mathbb{R}$. Then 
\begin{align*}
g \lsim h \iff \exists C>0,~ \forall \delta \in \Lambda \colon \quad g(\delta) \leq C h(\delta).
\end{align*}

The aim of the present contribution is to show that in contrast to the standard error estimates~\eqref{eq:rates_1} and \eqref{eq:rates_2}, improved convergence rates can be obtained for the error between the discrete minimisers $u_{\delta, \rm src}^{*}$ and~$u_{\delta, \rm eig}^{*}$ and the $(\cdot, \cdot)_X$-best approximation in $X_{\delta}$ of the exact minimisers $u^{*}_{\rm src}$ and~$u^{*}_{\rm eig}$. Roughly speaking, denoting by~$\Pi^{X}_\delta$ the $(\cdot, \cdot)_X$-orthogonal projection operator onto $X_{\delta}$, we will show that, under appropriate hypothesis,  for each $m \in \{0,1\}$, each ${\rm pb}\in \{{\rm src}, {\rm eig}\}$, and all ~$\delta$ small enough it holds that
\begin{align}\label{eq:rates_new}
 \Vert  u_{\delta, \rm pb}^{*} -\Pi^{X}_\delta u^*_{\rm pb}\Vert_{H^m(\Omega)} \leq \delta^{q(m)} \Vert u^*_{\rm pb} - \Pi^{X}_\delta u^*_{\rm pb} \Vert_{X},
\end{align}
where the exact rate $q(m)$ depends on the choice of discretisation (finite element or spectral) but is always larger than or equal to one and can even reach $q(1)=2$.

\vspace{1mm}
We remark here that the most comprehensive collection of results of this nature that we are familiar with is the article by Christopher Beattie \cite{beattie2000galerkin}, which further references prior work in \cite{chatelin2011spectral, babuvska1989finite, knyazev1997new}. However, the results contained in \cite{beattie2000galerkin} differ from ours in some important aspects. First, the analysis in \cite{beattie2000galerkin} concerns \emph{linear} eigenvalue problems-- albeit eigenvalue problems far more general than those of Schrödinger-type. Second, rather than studying the relationship between the Galerkin solution and the $(\cdot, \cdot)_X$-best approximation of the exact solution as we do in the present article, the analysis in \cite{beattie2000galerkin} studies the relationship between Galerkin eigenvectors and the projection of the exact solution onto the approximation space $X_{\delta}$ with respect to (i) the bilinear form induced induced by the underlying elliptic operator; (ii) the $L^2(\Omega)$ inner product. Finally, our results -- at least in the Fourier case (\textbf{Setting Two} for the eigenvalue problem) -- yield an improved convergence rate $q(0) =3$ in Estimate \eqref{eq:rates_new}, compared with only $2$ in \cite{beattie2000galerkin}.
\vspace{1mm}

Let us also briefly comment that while these results appear to be interesting, per se, independently of applications (for instance, as remarked in \cite{beattie2000galerkin}, the results for the eigenvalue problem belong to the `Galerkin folklore'), they are motivated in the present case by our proposal, based on {\sl a posteriori estimators}, of near-optimal strategies for calculating the numerical solution of a PDE to a precision fixed in advance. Non-linear periodic eigenvalues problems such as Equation \eqref{eq:1.2} with the plane-wave discretisation are frequently encountered in quantum chemistry and electronic structure calculations, and iterative algorithms to solve these problems -- known as Self-Consistant Field methods -- are therefore particularly relevant. In this context, following and in continuation of what is proposed in \cite{dusson2017posteriori}, we propose in \cite{hassan_2025} new error-balance strategies in order to achieve a given solution accuracy whilst minimising the computational cost. Our strategy in \cite{hassan_2025} involves first performing a large number of iterations in small-dimensional plane-wave bases (dealing with modes up to $n$ with $n \ll N$). Subsequently, the dimension of the plane-wave approximation spaces is increased in a controlled way to compute the higher modes of the eigenvector solution. The fact that the numerical solution displays superconvergence to the projection of the exact solution justifies this strategy. Indeed, performing a large number of iterations on small discrete problems (thus with small cost) provides a numerical solution whose low modes are very accurate. The additional iterations, in higher-dimensional spaces, are then concentrated purely on the high modes (i.e. on the orthogonal space to the first $n$ modes). This results in a much smaller \emph{effective} spectral radius of the underlying iteration matrix which leads to very few iterations being needed to attain the required accuracy. In \cite{hassan_2025}, strategies are proposed where only one iteration is required in the highest-dimensional space.


\vspace{1mm}
The remainder of this article is organised as follows. In the forthcoming Section \ref{sec:source}, we state and prove an abstract superconvergence result for general Galerkin discretisations of the \emph{source problem} \eqref{eq:1.1} that satisfy certain assumptions. Next, in Section \ref{sec:eig}, we state and prove an analogous abstract superconvergence result for general Galerkin discretisations of the \emph{eigenvalue problem} \eqref{eq:1.2} that satisfy certain assumptions. Subsequently, in Sections \ref{sec:4}, \ref{sec:5}, and \ref{sec:6}, we apply these abstract convergence results to conforming finite element discretisations, spectral polynomial, and spectral Fourier discretisations of the source and eigenvalue problems.

\section{An Abstract Superconvergence Result for the Source Problem}\label{sec:source}

Throughout this section, we assume that the computational domain $\Omega$, the function space $X$, the source function $f$, and the effective potential $V$ are chosen in accordance with either \textbf{Setting One} or \textbf{Setting Two}, and we denote by $u_{\rm src}^* \in X$ the unique minimiser of the continuous energy functional \eqref{eq:1.1} in $X$. Let us emphasise here that the space $X$ is equipped with its natural inner-product, i.e.,
\begin{equation}\label{eq:X_inner}
\begin{split}
 X=H_0^1(\Omega) \quad \implies \quad (u, v)_X&:= \int_{\Omega} \nabla u({\bf x})\cdot \nabla v({\bf x})\; d{\bf x} \qquad\hspace{2.38cm} \forall u, v \in X.\\ \nonumber
 X=H^1(\Omega) \quad \implies \quad (u, v)_X&:= \int_{\Omega} u({\bf x}) v({\bf x})\; d{\bf x}+\int_{\Omega} \nabla u({\bf x})\cdot \nabla v({\bf x})\; d{\bf x} \qquad \forall u, v \in X.
 \end{split}
\end{equation}

Our goal now is to state and prove some abstract convergence lemmas that are valid for any Galerkin approximation method that satisfies certain assumptions. We will then study, in the forthcoming Sections \ref{sec:4}-\ref{sec:6}, the implications of these abstract results for the concrete examples of the finite element, spectral polynomial and spectral Fourier methods.

Following the notation used earlier in Section \ref{sec:1}, let us denote by $\{X_{\delta}\}_{\delta >0}$ a sequence of conforming approximation spaces of $X$ that satisfy the approximation property \eqref{eq:approximation_property}. Next, for any $\delta >0 $, we denote by $u^*_{\delta, \rm src} \in X_\delta$ the unique minimiser of the discrete energy functional \eqref{eq:1.5} in $X_\delta$. As mentioned in the introduction, the existence of such a discrete minimiser (at least for~$\delta$ small enough) is guaranteed for the classical choices of finite element, spectral polynomial, and spectral Fourier approximation spaces $\{X_\delta\}_{\delta >0}$.\vspace{1mm}

\medskip
 We now assume that the following two conditions hold: \vspace{1mm}

\begin{description}
    \item[Assumption A.1 (Explicit Convergence Rates)]
    \end{description}
    We assume that there exists $r_{\rm src} \geq 2 $ such that for all $\delta>0$ sufficiently small it holds that
    \begin{align}\label{eq:assum_1a}
    \Vert u_{\rm src}^* - u^*_{\delta, \rm src} \Vert_{L^2(\Omega)} + \delta \Vert u_{\rm src}^* - u^*_{\delta, \rm src} \Vert_{X} \lsim\delta^{r_{\rm src}} \Vert u_{\rm src}^*\Vert_{H^{r_{\rm src}}(\Omega)}.
    \end{align}
    Additionally, we assume that the $(\cdot, \cdot)_X$-orthogonal projection operator $\Pi_\delta^X \colon X \rightarrow X_\delta$ satisfies for all $\delta>0$ sufficiently small and all $v \in H^2(\Omega)$ the estimate
    \begin{align}\label{eq:assum_1b}
    \Vert v -\Pi_\delta^X v \Vert_{L^2(\Omega)} \lsim \delta \Vert v -\Pi_\delta^X v \Vert_{X} \lsim \delta^{2} \Vert v\Vert_{H^{2}(\Omega)}.
    \end{align}

    \vspace{1mm}

\begin{description}
    \item[Assumption A.2 (Uniform Boundedness of Discrete Minimisers)]
    \end{description}
    We assume that the sequence of discrete constrained minimisers $\{u^{*}_{\delta, \rm src}\}_{\delta >0}$ is uniformly bounded in $L^{\infty}(\Omega)$, i.e., there exists a $C>0$ such that for all $\delta >0$ it holds that
    \begin{align}\label{eq:assum_3}
\Vert u^{*}_{\delta, \rm src} \Vert_{L^{\infty}(\Omega)} \leq C.
    \end{align}

    \vspace{1mm}
    
\begin{remark}[A Sufficient Condition to Guarantee \textbf{Assumption A.2}]\label{rem:inverse}~

\noindent    It is pertinent to remark here that for approximation spaces $\{X_{\delta}\}_{\delta >0}$ that consist of more regular functions, for instance, if $X_{\delta} \subset H^2(\Omega)$ (which is \emph{not} the case for $\mathcal{P}_1$ finite elements), \textbf{Assumption A.2} above can be replaced with the following two hypotheses:
    \begin{itemize}
        \item First, we assume the existence of a suitable optimal-rate projection operator. More precisely, we assume that there exists $ s^*_{\rm src}\geq2$ such that for every $\delta >0$ sufficiently small there exists an operator $\Pi_\delta^{2} \colon H^2(\Omega)\cap X \rightarrow X_\delta$ such that for all real numbers $0\leq t \leq 2$ and any $v \in H^{s^*_{\rm src}}(\Omega) \cap X$ it holds that
        \begin{align}\label{eq:assum_3primeprime}
            \Vert v - \Pi_\delta^{2} v \Vert_{H^{t}(\Omega)} \lsim \delta^{s^*_{\rm src}-t} \Vert v \Vert_{H^{s^*_{\rm src}}(\Omega)}.
        \end{align}

        \item Second, we assume the existence of a suitable inverse inequality, i.e., we assume that there exists a mapping $\theta \colon [1, +\infty) \rightarrow \mathbb{R}$ such that 
    \begin{equation}\label{eq:assum_3prime}
    \begin{split}
        \forall p \geq 1, ~ \forall v_\delta \in X_\delta \cap H^{p}(\Omega)\colon& \qquad \Vert v_\delta \Vert_{H^{p}(\Omega)} \lsim \delta^{-\theta(p)} \Vert  v_\delta \Vert_{H^1(\Omega)} \qquad \text{and}\\
\exists p^* \in (d/2, 2)\colon& \qquad \theta(p^*)+1- \min\{s^*_{\rm src}, r_{\rm src}\}< 0.
    \end{split}
    \end{equation}
    \end{itemize}

\noindent Indeed, we observe that we can write
\begin{align*}
    \Vert u^*_{\rm src}- u^*_{\delta, \rm src} \Vert_{H^{p^*}(\Omega)} &\leq \Vert u^*_{\rm src}-  \Pi_\delta^{2} u^*_{\rm src} \Vert_{H^{p^*}(\Omega)} + \Vert  \Pi_\delta^{2} u^*_{\rm src}- u^*_{\delta, \rm src} \Vert_{H^{p^*}(\Omega)}\\
    &\lsim \delta^{s_{\rm src}^*-p^*} \Vert u^*_{\rm src}\Vert_{H^{s_{\rm src}^*}(\Omega)}  + \Vert  \Pi_\delta^{2} u^*_{\rm src}- u^*_{\delta, \rm src} \Vert_{H^{p^*}(\Omega)}\\
    &\lsim \delta^{s_{\rm src}^*-p^*} \Vert u^*_{\rm src}\Vert_{H^{s_{\rm src}^*}(\Omega)}  + \delta^{-\theta(p^*)} \Vert  \Pi_\delta^{2} u^*_{\rm src}- u^*_{\delta, \rm src} \Vert_{H^{1}(\Omega)}\\
    &\leq \delta^{s_{\rm src}^*-p^*} \Vert u^*_{\rm src}\Vert_{H^{s_{\rm src}^*}(\Omega)}  + \delta^{-\theta(p^*)} \Vert  \Pi_\delta^{2} u^*_{\rm src}- u^*_{\rm src} \Vert_{H^{1}(\Omega)}+ \delta^{-\theta(p^*)} \Vert u^*_{\rm src}- u^*_{\delta, \rm src} \Vert_{H^{1}(\Omega)}\\
    &\lsim \delta^{s_{\rm src}^*-p^*} \Vert u^*_{\rm src}\Vert_{H^{s_{\rm src}^*}(\Omega)}  + \delta^{s_{\rm src}^*-1-\theta(p^*)} \Vert  u^*_{\rm src}\Vert_{H^{s_{\rm src}^*}(\Omega)}+ \delta^{r_{\rm src}^*-\theta(p^*)} \Vert u^*_{\rm src}\Vert_{H^{s_{\rm src}^*}(\Omega)}.
\end{align*}
In the above calculation, the second step follows from Estimate \eqref{eq:assum_3primeprime}, the third step follows from Estimate~\eqref{eq:assum_3prime} and the last step follows from both Estimate \eqref{eq:assum_3primeprime} and \eqref{eq:assum_3prime}. 

Using now the fact that $d/2 < p^* < 2 \leq s_{\rm src}^*$ and $\theta(p^*)+1- \min\{s^*_{\rm src}, r_{\rm src}\}< 0$ by assumption and recalling the Sobolev embedding $H^{t}(\Omega)  \hookrightarrow L^{\infty}(\Omega)$ for all $t > d/2$, we deduce that $u^*_{\delta, \rm src}  \to u^*_{\rm src}$ in the $L^{\infty}(\Omega)$-norm. The uniform boundedness in $L^{\infty}(\Omega)$ of the sequence $\{u^*_{\delta, \rm src} \}_{\delta >0}$ now follows. 
\end{remark}

\vspace{1mm}
As the next step, we introduce two sequences of auxiliary potential functions that will play a key role in our analysis.
\begin{definition}\label{def:v_aux}
    Assume the framework adopted in the current section and recall that $V \in H^{r_v}(\Omega), ~ r_v > d/2$ denotes the effective potential, $u_{\rm src}^* \in X$ denotes the exact minimiser of the energy functional \eqref{eq:1.1}, and for each $\delta >0$, $u^*_{\delta, \rm src} \in X_\delta$ denotes the discrete minimiser of the discrete energy functional \eqref{eq:1.5}. We define the auxiliary potential functions ${V}_{\delta} \colon \Omega \rightarrow \mathbb{R}$ and $\widetilde{V}_{\delta} \colon \Omega \rightarrow \mathbb{R}$ as
    \begin{equation}\label{eq:v_aux}
        {V}_{\delta} := V + \left(u_{\rm src}^* \right)^2+ \left(u^*_{\delta, \rm src} \right)^2 + u^*_{\rm src}u^*_{\delta, \rm src}, \qquad \widetilde{V}_{\delta} := V_{\delta} - 1.
    \end{equation}
\end{definition}

As a direct consequence of \textbf{Assumptions A.1} and \textbf{A.2}, we can deduce the following regularity result on the auxiliary potentials $\{{V}_{\delta}\}_{\delta >0}$ and $\{\widetilde{V}_{\delta}\}_{\delta >0}$. 
\begin{lemma}\label{lem:combine_1}
    Assume the framework adopted in the current section, let \textbf{Assumptions A.1} and~\textbf{A.2} hold, and for any $\delta>0$ let the auxiliary potential functions ${V}_{\delta} \colon \Omega \rightarrow \mathbb{R}$ and $\widetilde{V}_{\delta} \colon \Omega \rightarrow \mathbb{R}$ be defined as in Definition \ref{def:v_aux}. Then 
\begin{align*} 
    \{{V}_{\delta}\}_{\delta >0}, ~ \{\widetilde{V}_{\delta}\}_{\delta >0} ~\text{ are uniformly bounded in } L^{\infty}(\Omega).
\end{align*}
\end{lemma}
\begin{proof}
Since the exact minimiser $u^*_{\rm src} \in H^2(\Omega)$ (recall the inclusion \eqref{eq:elliptic_regularity}), the result follows from the Sobolev embedding $H^{s}(\Omega)  \hookrightarrow L^{\infty}(\Omega)$ for all $s > d/2$ and \textbf{Assumption A.2}. 
\end{proof}

Lemma \ref{lem:combine_1} together with \textbf{Assumptions A.1} and \textbf{A.2} is sufficient to prove a first abstract superconvergence result in the $\Vert \cdot \Vert_X$-norm. However, in order to obtain an analogous abstract superconvergence result in the $\Vert \cdot \Vert_{L^2(\Omega)}$-norm -- which also allows the further improvement of the $\Vert \cdot \Vert_X$ result -- we require an additional assumption. We shall demonstrate in the sequel that this assumption is satisfied by higher-order conforming finite element, spectral polynomial and spectral Fourier methods. Unfortunately, unlike \textbf{Assumptions A.1} and \textbf{A.2}, we are unable to prove that this third assumption is satisfied by the linear, conforming finite element method. \vspace{3mm}

\begin{description}
    \item[Assumption A.3 (A Convergence Result on the Auxiliary Potential)]
    \end{description}
  For any $\delta>0$, let the auxiliary potential function $\widetilde{V}_{\delta} \colon \Omega \rightarrow \mathbb{R}$ be defined as in Definition \ref{def:v_aux}, and let $\Pi^X_\delta \colon X \rightarrow X_\delta$ denote the $(\cdot, \cdot)_X$-orthogonal projection operator onto $X_\delta$. We assume that there exists a real number $q^*_{\rm src}\in (0, 2]$ such that for all $\delta$ sufficiently small, all $\varphi\in H^2(\Omega)\cap X$ and all $\chi \in X$ it holds that
    \begin{align}\label{eq:assum_4b}
    \int_\Omega \widetilde{V}_{\delta}({\bf x})  \varphi({\bf x}) [\chi- \Pi^X_\delta \chi]({\bf x})\; d{\bf{x}} &\lsim \delta^{q^*_{\rm src}+1}\Vert \varphi \Vert_{H^2(\Omega)} \Vert  \chi- \Pi^X_\delta \chi\Vert_{X}.
\end{align}


\medskip

Equipped with \textbf{Assumptions A.1, A.2} and possibly \textbf{A.3}, we can now prove the following abstract superconvergence result.

\medskip

\begin{lemma}\label{lem:combine_2}
    Assume the framework adopted in the current section, let \textbf{Assumptions A.1} and~\textbf{A.2} hold, and for every $\delta>0$ let $\Pi^X_\delta \colon X \rightarrow X_\delta$ denote the $(\cdot, \cdot)_X$-orthogonal projection operator onto $X_\delta$. Then we have the error estimate
    \begin{align}\label{eq:source_lower_rates}
      \Vert \Pi_\delta^X u^*_{\rm src}-u_{\delta, \rm src}^{*}\Vert_{X} &\lsim   \delta\Vert \Pi^X_{\delta} u^*_{\rm src}-u^*_{\rm src}\Vert_{X}.
    \end{align}
    If, in addition, \textbf{Assumption A.3} also holds, then we have the improved error estimate
    \begin{align}\label{eq:source_higher_rates}
      \Vert \Pi_\delta^X u^*_{\rm src}-u_{\delta, \rm src}^{*}\Vert_{L^2(\Omega)} + \delta^{\frac{q^*_{\rm src}}{2}}\Vert \Pi_\delta^X u^*_{\rm src}-u_{\delta, \rm src}^{*}\Vert_{X} &\lsim   \delta^{q^*_{\rm src}+1}\Vert \Pi_\delta^X u_{\rm src}-u^*_{\rm src}\Vert_{X},
    \end{align}
    where $q^*_{\rm src}>0$ is the constant introduced in \textbf{Assumption A.3}.
\end{lemma}
\begin{proof}
     We begin by introducing, for any $\delta>0$, the bilinear form $a_\delta \colon X \times X \rightarrow \R$ defined as
\begin{align}\label{eq:Spectral_new_bilinear}
	\forall v, w \in X\colon \quad     a_\delta(v, w):=& \int_\Omega \nabla v({\bf x}) \cdot \nabla w({\bf x}) d{\bf x}+ \int_\Omega  {V}_{\delta}({\bf x})v({\bf x}) w({\bf x}) d{\bf x}, 
\end{align}
where the auxiliary potential ${V}_{\delta}= V + \big(u^*_{\rm src}\big)^2 + \big(u^*_{\delta, \rm src}\big)^2 + u^*_{\rm src}u^*_{\delta, \rm src}$ has been introduced through Equation \eqref{eq:v_aux} in Definition \ref{def:v_aux}.  

It follows from Lemma \ref{lem:combine_1} that $\{{V}_{\delta}\}_{\delta >0}$ is uniformly bounded in $L^{\infty}(\Omega)$ which immediately yields that the bilinear form $a_\delta \colon X \times X \rightarrow \mathbb{R}$ is continuous with uniformly bounded (w.r.t $\delta$) continuity constant. Since, we can additionally express ${V}_{\delta}$ in the form
\begin{align*}
    {V}_{\delta}= V + \frac{1}{2}\Bigl[\big(u^*_{\rm src}\big)^2 + \big(u^*_{\delta, \rm src}\big)^2 + \big(u^*_{\rm src}+u^*_{\delta, \rm src}\big)^2\Bigr]
\end{align*}
we also have that each ${V}_{\delta } \ge V$, from which we deduce that $a_\delta$ is coercive on $X$ with coercivity constant $\alpha >0$ independent of $\delta$.

Next, we claim that the bilinear form $a_\delta$ satisfies a Galerkin orthogonality~relation for every $\delta >0$. To see this, note that the exact minimiser $u^*_{\rm src}$ and discrete minimiser $u^*_{\delta, \rm src}$ satisfy the following Euler-Lagrange equations:
\begin{align*}
\forall v \in X\colon &\qquad (\nabla u^*_{\rm src},\nabla v)_{L^2(\Omega)}+  (V u^*_{\rm src}, v )_{L^2(\Omega)} +\big((u^*_{\rm src})^3, v \big)_{L^2(\Omega)}&&= (f, v)_{L^2(\Omega)},\\[1em]
\forall v_{\delta} \in X_{\delta}\colon& \qquad (\nabla u_{\delta, \rm src}^*,\nabla v_{\delta})_{L^2(\Omega)}+  (V u_{\delta, \rm src}^*, v_{\delta} )_{L^2(\Omega)} +\big((u_{\delta, \rm src}^*)^3, v_{\delta} \big)_{L^2(\Omega)}&&= (f, v_{\delta})_{L^2(\Omega)}.
\end{align*}
Consequently, for all $v_{\delta} \in X_{\delta}$ we must have
\begin{align*}
    0&=(\nabla u^*_{\rm src}- \nabla u_{\delta, \rm src}^*,\nabla v_{\delta})_{L^2(\Omega)}+  \big((u^*_{\rm src}-u_{\delta, \rm src}^*)V, v_{\delta} \big)_{L^2(\Omega)} +\big((u^*_{\rm src})^3-(u_{\delta, \rm src}^*)^3, v_{\delta} \big)_{L^2(\Omega)}\\
&=(\nabla u^*_{\rm src}- \nabla u_{\delta, \rm src}^*,\nabla v_{\delta})_{L^2(\Omega)}+  \Big((u^*_{\rm src}-u_{\delta, \rm src}^*)\big(V+ (u_{\rm src}^*)^2+ (u^*_{\delta, \rm src})^2 + u^*_{\rm src}u^*_{\delta, \rm src}\big), v_{\delta} \Big)_{L^2(\Omega)}. 
\end{align*}

Recalling now Definition \ref{def:v_aux} of the auxiliary potential~$V_{\delta}, ~ \delta >0$ and the definition of the bilinear form~$a_{\delta}$ given by Equation \eqref{eq:Spectral_new_bilinear}, we deduce the Galerkin orthogonality relation
\begin{align}\label{eq:1_poly_pre}
\forall v_\delta \in X_{\delta}\colon \qquad a_\delta(u^*_{\rm src}- u_{\delta, \rm src}^{*}, v_\delta)&=0, \qquad \text{so that}\\
a_\delta(u^*_{\rm src}-\Pi_\delta^{X} u^*_{\rm src}, u^*_{\delta, \rm src}-\Pi_\delta^{X} u^*_{\rm src})&= a_\delta(u^*_{\delta, \rm src}-\Pi_\delta^{X} u^*_{\rm src}, u^*_{\delta, \rm src}-\Pi_\delta^{X} u^*_{\rm src}). \label{eq:1_poly}
\end{align}
Since the projection operator $\Pi_\delta^{X}$ is, by definition, $(\cdot, \cdot)_X$-orthogonal, we can further deduce that
\begin{align}\nonumber
a_\delta(u^*_{\rm src}-\Pi_\delta^{X} u^*_{\rm src}, u^*_{\delta, \rm src}-\Pi_\delta^{X} u^*_{\rm src}) =& \int_\Omega \widetilde{V}_{\delta}({\bf x}) [u^*_{\rm src}-\Pi_\delta^{X} u^*_{\rm src}] ({\bf x}) [u^*_{\delta, \rm src}-\Pi_\delta^{X} u^*_{\rm src}]({\bf x})\; d{\bf{x}}\\[0.5em]
\le& \| \widetilde{V}_{\delta}\|_{L^\infty(\Omega)} \|u^*_{\rm src}-\Pi_\delta^{X} u^*_{\rm src}\|_{L^2(\Omega)} \|u^*_{\delta, \rm src}-\Pi_\delta^{X} u^*_{\rm src}\|_{L^2(\Omega)}.\label{555b}
\end{align}

Combining now Equation \eqref{eq:1_poly} with Estimate \eqref{555b} and using the coercivity of $a_\delta$ we obtain
\begin{equation} \label{eq:30ab}
\alpha \| u^*_{\delta, \rm src}-\Pi_\delta^{X} u^*_{\rm src} \|_{X}^2 \le \| \widetilde{V}_{\delta}\|_{L^\infty(\Omega)} \|u^*_{\rm src}-\Pi_\delta^{X} u^*_{\rm src}\|_{L^2(\Omega)} \|u^*_{\delta, \rm src}-\Pi_\delta^{X} u^*_{\rm src}\|_{L^2(\Omega)}.
\end{equation}

In order to prove the convergence estimate \eqref{eq:source_lower_rates}, we make use of the convergence estimate~\eqref{eq:assum_1b} postulated in~\textbf{Assumption~A.1} to deduce that for all $\delta >0$ it holds that
\begin{equation*}
\alpha \| u^*_{\delta, \rm src}-\Pi_\delta^{X} u^*_{\rm src} \|_{X}^2 \le \delta \| \widetilde{V}_{\delta}\|_{L^\infty(\Omega)}  \|u^*_{\rm src}-\Pi_\delta^{X} u^*_{\rm src}\|_{X}  \|u^*_{\delta, \rm src}-\Pi_\delta^{X} u^*_{\rm src}\|_{X}.
\end{equation*}
 The convergence estimate \eqref{eq:source_lower_rates} now follows by recalling the uniform boundedness in $L^{\infty}(\Omega)$ of $\{\widetilde{V}_{\delta}\}_{\delta >0}$ yielded by Lemma~\ref{lem:combine_1}.

\vspace{2mm}
To prove the improved convergence estimate \eqref{eq:source_higher_rates}, we return to Inequality \eqref{eq:30ab} and notice that we can write  
\begin{align}\nonumber
\alpha \| u^*_{\delta, \rm src}-\Pi_\delta^{X} u^*_{\rm src} \|_{X}^2 &\le \| \widetilde{V}_{\delta}\|_{L^\infty(\Omega)} \|u^*_{\rm src}-\Pi_\delta^{X} u^*_{\rm src}\|_{L^2(\Omega)} \|u^*_{\delta, \rm src}-\Pi_\delta^{X} u^*_{\rm src}\|_{L^2(\Omega)}\\
 &\le \| \widetilde{V}_{\delta}\|_{L^\infty(\Omega)} \Bigl[  \|u^*_{\rm src}-u^*_{\delta, \rm src}\|_{L^2(\Omega)}+\|u^*_{\delta, \rm src}-\Pi_\delta^{X} u^*_{\rm src}\|_{L^2(\Omega)} \Bigr]\|u^*_{\delta, \rm src}-\Pi_\delta^{X} u^*_{\rm src}\|_{L^2(\Omega)}.\label{eq:30b}
\end{align}

We first bound the term $ \Vert u^*_{\delta, \rm src}-\Pi_\delta^{X}  u^*_{\rm src}\Vert_{L^2(\Omega)} $. To do so, we introduce for any $\delta>0$, the following adjoint problem: we seek $ \vartheta\in X $ such that for all $v \in X$ it holds that
	\begin{equation}\label{def:adjoint_3}
	a_\delta(v, \vartheta)=\int_\Omega v({\bf{x}})[u^*_{\delta, \rm src}-\Pi_\delta^{X} u^*_{\rm src}](\bf{x} ) \; d{\bf{x}}.
	\end{equation}	
It follows from the coercivity and continuity of the bilinear form $a_\delta$ that Equation \eqref{def:adjoint_3} is well-posed. Since, additionally, $u^*_{\delta, \rm src}-\Pi_\delta^{X} u^*_{\rm src} \in L^2(\Omega)$ and the auxiliary potential ${V}_{\delta} \in L^{\infty}(\Omega)$, it follows from standard elliptic regularity theory and a bootstrapping argument that $\vartheta \in H^2(\Omega)$ and satisfies
	\begin{equation}\label{eq:bound_H_2_2_poly}
	\Vert \vartheta \Vert_{H^2(\Omega)}\lsim \Vert u^*_{\delta, \rm src}-\Pi_\delta^{X}  u^*_{\rm src}\Vert_{L^2(\Omega)}.
	\end{equation}
	Combining now the symmetry of $a_\delta$ and the Galerkin orthogonality property from Equation~\eqref{eq:1_poly_pre}, we deduce from the adjoint equation \eqref{def:adjoint_3} that
\begin{align*}
    \Vert u^*_{\delta, \rm src}-\Pi_\delta^{X} u^*_{\rm src}\Vert_{L^2(\Omega)}^2 =a_\delta(\vartheta, u^*_{\delta, \rm src}-\Pi_\delta^{X} u^*_{\rm src})= a_\delta(\vartheta-\Pi_\delta^{X} \vartheta, u^*_{\delta, \rm src}-\Pi_\delta^{X} u^*_{\rm src})+a_\delta(\Pi_\delta^{X} \vartheta, u^*_{\rm src}-\Pi_\delta^{X} u_{\rm src}^*).
\end{align*} 
Appealing once again to the $(\cdot, \cdot)_X$-orthogonality of the projection operator $\Pi_\delta^{X}$, we obtain that
\begin{align}\nonumber
\Vert u^*_{\delta, \rm src}-\Pi_\delta^{X} u^*_{\rm src}\Vert_{L^2(\Omega)}^2 &=  \int_\Omega \widetilde{V}_{\delta}({\bf x})  [\vartheta-\Pi_\delta^{X} \vartheta] [u^*_{\delta, \rm src}-\Pi_\delta^{X} u^*_{\rm src}]({\bf x})\; d{\bf{x}}\\ \nonumber
&+\int_\Omega \widetilde{V}_{\delta}({\bf x}) [\Pi_\delta^{X}  \vartheta]({\bf{x}})[u^*_{\rm src}-\Pi_\delta^{X} u^*_{\rm src}]({\bf x})\; d{\bf{x}}\\[0.5em] \nonumber
&=\underbrace{\int_\Omega \widetilde{V}_{\delta}({\bf x})  [\vartheta-\Pi_\delta^{X} \vartheta] [u^*_{\delta, \rm src}-\Pi_\delta^{X} u^*_{\rm src}]({\bf x})\; d{\bf{x}}}_{:= {\rm (I)}}\\ \nonumber
&-\underbrace{\int_\Omega \widetilde{V}_{\delta}({\bf x})[\vartheta-\Pi_\delta^{X} \vartheta]({\bf{x}})[u^*_{\rm src}-\Pi_\delta^{X} u^*_{\rm src}]({\bf x})\; d{\bf{x}}}_{:= {\rm (II)}}\\
&+\int_\Omega \widetilde{V}_{\delta}({\bf x})\vartheta({\bf{x}})[u^*_{\rm src}-\Pi_\delta^{X} u^*_{\rm src}]({\bf x})\; d{\bf{x}}. \label{er1b}
\end{align}

To bound Estimate \eqref{er1b}, we first make use of Hölder's inequality, Equation \eqref{eq:assum_1a} from \textbf{Assumption A.1}, and Estimate \eqref{eq:bound_H_2_2_poly} to deduce that
\begin{align}\nonumber
    {\rm (I)+ (II)}=&\int_\Omega \widetilde{V}_{\delta}({\bf x})  [\vartheta-\Pi_\delta^{X} \vartheta] (\bold{x})[u^*_{\delta, \rm src}-\Pi_\delta^{X} u^*_{\rm src}]({\bf x})\; d{\bf{x}}-\int_\Omega \widetilde{V}_{\delta}({\bf x})[\vartheta-\Pi_\delta^{X} \vartheta]({\bf{x}})[u^*_{\rm src}-\Pi_\delta^{X} u^*_{\rm src}]({\bf x})\; d{\bf{x}}\\[1em] \nonumber
    \leq& \Vert \widetilde{V}_{\delta}\Vert_{L^{\infty}(\Omega)}  \Vert \vartheta-\Pi_N^{X} \vartheta\Vert_{L^2(\Omega)} \big(\Vert u^*_{\delta, \rm src}-\Pi_\delta^{X} u^*_{\rm src}\Vert_{L^2(\Omega)}+\Vert u^*_{\rm src}-\Pi_\delta^{X} u^*_{\rm src}\Vert_{L^2(\Omega)}\big)\\[1em] \nonumber
    \lsim& \delta^{2}\Vert \widetilde{V}_{\delta}\Vert_{L^{\infty}(\Omega)}  \Vert u^*_{\delta, \rm src}-\Pi_\delta^{X} u^*_{\rm src}\Vert_{L^2(\Omega)} \big(\Vert u^*_{\delta, \rm src}-\Pi_\delta^{X} u^*_{\rm src}\Vert_{L^2(\Omega)}+\Vert u^*_{\rm src}-\Pi_\delta^{X} u^*_{\rm src}\Vert_{L^2(\Omega)}\big)\\[1em] \label{eq:mistake}
    \lsim& \delta^{2}\Vert \widetilde{V}_{\delta}\Vert_{L^{\infty}(\Omega)} \big(\Vert u^*_{\delta, \rm src}-\Pi_\delta^{X} u^*_{\rm src}\Vert^2_{L^2(\Omega)}+\Vert u^*_{\delta, \rm src}-\Pi_\delta^{X} u^*_{\rm src}\Vert_{L^2(\Omega)}\Vert u^*_{\rm src}-\Pi_\delta^{X} u^*_{\rm src}\Vert_{L^2(\Omega)}\big).
\end{align}
Combining now Inequality \eqref{eq:mistake} with the prior Estimate \eqref{er1b} and recalling the uniform boundedness of $\{\widetilde{V}_{\delta}\}_{\delta >0}$ in $L^{\infty}(\Omega)$ as a corollary of Lemma \ref{lem:combine_1}, we deduce that 
\begin{align*}
    \Vert u^*_{\delta, \rm src}-\Pi_\delta^{X} u^*_{\rm src}\Vert_{L^2(\Omega)}^2 &\lsim  \delta^{2}\Vert u^*_{\delta, \rm src}-\Pi_\delta^{X} u^*_{\rm src}\Vert^2_{L^2(\Omega)}\\[0.5em]
    &+\delta^{2}\Vert u^*_{\delta, \rm src}-\Pi_\delta^{X} u^*_{\rm src}\Vert_{L^2(\Omega)}\Vert u^*_{\rm src}-\Pi_\delta^{X} u^*_{\rm src}\Vert_{L^2(\Omega)}\\[0.5em]
    &+\int_\Omega \widetilde{V}_{\delta}({\bf x})\vartheta({\bf{x}})[u^*_{\rm src}-\Pi_\delta^{X} u^*_{\rm src}]({\bf x})\; d{\bf{x}}.
\end{align*}
Notice now that the norm estimate $ \Vert u^*_{\delta, \rm src}-\Pi_\delta^{X} u^*_{\rm src}\Vert^2_{L^2(\Omega)}$ appears on the right-hand side of the above inequality with pre-factor $\delta^2$. Consequently, for all $\delta$ sufficiently small, this norm estimate can be subsumed by its counterpart appearing on the left-hand side of the above inequality (which has pre-factor independent of $\delta$). It follows that for all $\delta$ sufficiently small, we must have
\begin{align}\nonumber
\Vert u^*_{\delta, \rm src}-\Pi_\delta^{X} u^*_{\rm src}\Vert_{L^2(\Omega)}^2 &\lsim \delta^{2} \Vert u^*_{\delta, \rm src}-\Pi_\delta^{X} u^*_{\rm src}\Vert_{L^2(\Omega)}\Vert u^*_{\rm src}-\Pi_\delta^{X} u^*_{\rm src}\Vert_{L^2(\Omega)}\\ \label{eq:mistake_2}
&+ \int_\Omega \widetilde{V}_{\delta}({\bf x})\vartheta({\bf{x}})[u^*_{\rm src}-\Pi_\delta^{X} u^*_{\rm src}]({\bf x})\; d{\bf{x}}.
\end{align}
To bound the second term appearing on the right-hand side of Inequality \eqref{eq:mistake_2}, we make use of Estimate \eqref{eq:assum_4b} in \textbf{Assumption A.3} to deduce the existence of a real number $q^*_{\rm src}\in (0, 2]$ such that for all $\delta$ sufficiently small it holds that
\begin{align}\nonumber
\Vert u^*_{\delta, \rm src}-\Pi_\delta^{X} u^*_{\rm src}\Vert_{L^2(\Omega)}^2 &\lsim \delta^{2}\Vert u^*_{\delta, \rm src}-\Pi_\delta^{X} u^*_{\rm src}\Vert_{L^2(\Omega)}\Vert u^*_{\rm src}-\Pi_\delta^{X} u^*_{\rm src}\Vert_{L^2(\Omega)}\\ \nonumber
&+\delta^{q^*_{\rm src}+1}\Vert \vartheta \Vert_{H^2(\Omega)} \Vert u^*_{\rm src}-\Pi_\delta^{X} u^*_{\rm src}\Vert_{X)}\\[0.5em] \nonumber
&\lsim \delta^{2} \Vert u^*_{\delta, \rm src}-\Pi_\delta^{X} u^*_{\rm src}\Vert_{L^2(\Omega)}\Vert u^*_{\rm src}-\Pi_\delta^{X} u^*_{\rm src}\Vert_{L^2(\Omega)}\\ \nonumber
&+ \delta^{q^*_{\rm src}+1}\Vert  u^*_{\delta, \rm src}-\Pi_\delta^{X}  u^*_{\rm src}\Vert_{L^2(\Omega)} \Vert u^*_{\rm src}-\Pi_\delta^{X} u^*_{\rm src}\Vert_{X}\\ \nonumber
&\lsim \delta^{3} \Vert u^*_{\delta, \rm src}-\Pi_\delta^{X} u^*_{\rm src}\Vert_{L^2(\Omega)}\Vert u^*_{\rm src}-\Pi_\delta^{X} u^*_{\rm src}\Vert_{X}\\ \nonumber
&+ \delta^{q^*_{\rm src}+1}\Vert  u^*_{\delta, \rm src}-\Pi_\delta^{X}  u^*_{\rm src}\Vert_{L^2(\Omega)} \Vert u^*_{\rm src}-\Pi_\delta^{X} u^*_{\rm src}\Vert_{X}, \nonumber
\end{align}
where the second step follows from Inequality \eqref{eq:bound_H_2_2_poly} and the third step follows from Estimate~\eqref{eq:assum_1b} in \textbf{Assumption A.1}. Since $q_{\rm src}^* \in (0, 2]$, we finally deduce the sought-after estimate
\begin{align}\label{eq:Has_new_3}
\Vert u^*_{\delta, \rm src}-\Pi_\delta^{X} u^*_{\rm src}\Vert_{L^2(\Omega)} &\lsim \delta^{q^*_{\rm src}+1} \Vert  u^*_{\rm src}-\Pi_\delta^{X}  u^*_{\rm src}\Vert_{X}.
\end{align}

To complete the proof and obtain also the improved convergence rates with respect to the $\Vert \cdot \Vert_{X}$-norm, we simply insert Estimate \eqref{eq:Has_new_3} into the earlier $\Vert \cdot \Vert_{X}$-norm estimate~\eqref{eq:30b} and make use of the convergence rates postulated in \textbf{Assumption A.1}.

\end{proof}

We end this section by remarking that if, in the Aubin-Nitsche trick, a full order $\delta$ of improvement is not achieved between the $\Vert \cdot \Vert_{L^2(\Omega)}$ and $\Vert \cdot \Vert_{X}$ convergence rates (see, e.g., Estimates \eqref{eq:assum_1a}-\eqref{eq:assum_1b}), then the conclusions of Lemma \ref{lem:combine_2} are still valid in the sense that the convergence of $\Pi_\delta^X u^*_{\rm src}-u_{\delta, \rm src}^{*}$ in both the $\Vert \cdot \Vert_{L^2(\Omega)}$-norm and $\Vert \cdot \Vert_{X}$-norm is improved, albeit not at the same order.

\section{An Abstract Superconvergence Result for the Constrained Minimisation Problem}\label{sec:eig}

Throughout this section, we assume that the computational domain $\Omega$, the function space $X$, and the effective potential $V$ are chosen in accordance with either \textbf{Setting One} or \textbf{Setting Two}, and we denote by $u_{\rm eig}^* \in X$ the unique positive minimiser to the constrained minimisation problem~\eqref{eq:1.2}. As in Section \ref{sec:source}, the space $X$ is equipped with its natural inner-product (see Equation \eqref{eq:X_inner}).

Following the approach of Section \ref{sec:source}, the aim of this section is to state and prove some abstract convergence lemmas that are valid for any Galerkin approximation of the constrained minimisation problem~\eqref{eq:1.2} that satisfies certain assumptions. We will then apply, in the forthcoming Sections \ref{sec:4}-\ref{sec:6}, this abstract theory to the concrete examples of the conforming finite element method and spectral polynomial and spectral Fourier methods. Before proceeding with this analysis however, it will be useful to introduce some additional notation and recall some basic properties of the constrained minimisation problem~\eqref{eq:1.2}. As usual, the results in this section will require the notion of function regularity in terms of the usual Sobolev scale, and we refer to, e.g., \cite[Chapter 3]{mclean2000strongly} for precise definitions of these spaces.

To begin with, we introduce, for any $v \in L^{3/2}(\Omega)$, the operator $\mathcal{A}_v \colon X \rightarrow X^*$ defined~as
\begin{align*}
    \forall z, w \in X \colon \quad \langle \mathcal{A}_v z,  w\rangle_{X^* \times X}:= ( \nabla z, \nabla w)_{L^2(\Omega} + (Vz,w)_{L^2(\Omega)} + (v^2z,w)_{L^2(\Omega},
\end{align*}
where $\langle \cdot, \cdot\rangle_{X^* \times X}$ denotes the duality pairing between $X$ and its dual space $X^*$. It is easy to check that, under the assumptions we have imposed on the effective potential $V$, the operator $\mathcal{A}_v$ is bounded below and above for any $v \in X\cap H^{d/2 + \epsilon}(\Omega), ~ \epsilon >0$. Since the exact minimiser $u^*_{\rm eig} \in X \cap H^2(\Omega)$ as a consequence of elliptic regularity theory and a boot-strapping argument (recall Estimate \eqref{eq:elliptic_regularity}), the operator $\mathcal{A}_{u^*_{\rm eig}}$ is also bounded below and above. 

It is straightforward to verify that the exact minimiser $u^*_{\rm eig} \in X$ satisfies the eigenvalue problem
\begin{equation}\label{eq:Euler_Lagrange_continuous}
    \forall v \in X \colon \qquad \langle \mathcal{A}_{u^*_{\rm eig}} u^*_{\rm eig}, v \rangle_{X^* \times X}= \lambda^*(u^*_{\rm eig}, v)_{L^2(\Omega)}.    
\end{equation}
Indeed, Equation \eqref{eq:Euler_Lagrange_continuous} is nothing else than the Euler-Lagrange equation associated with the constrained minimisation problem \eqref{eq:1.2}. Let us emphasise here that $\lambda^*\in \mathbb{R}$ is the lowest eigenvalue of the operator $\mathcal{A}_{u^*_{\rm eig}}$ (see, e.g., \cite{dusson2017posteriori}).

It is known (again, see, e.g., \cite{dusson2017posteriori}) that the energy functional $\mathcal{E}_{\rm eig}$ introduced in Equation \eqref{eq:1.2} is twice continuously differentiable with second derivative $\mathcal{E}_{\rm eig}''[v] \colon X \rightarrow X^*$  at $v \in X$  given by
\begin{equation}\label{eq:energy_derivative}
\begin{split}
    \forall z, w \in X \colon \quad \langle \mathcal{E}_{\rm eig}''[v]z, w\rangle_{X^* \times X} &= ( \nabla z, \nabla w)_{L^2(\Omega)} + (Vz,w)_{L^2(\Omega)} + 3(v^2z,w)_{L^2(\Omega)}\\
    &=\langle \mathcal{A}_v z,  w\rangle_{X^* \times X}+ 2 (v^2z,w)_{L^2(\Omega)}.
\end{split}
\end{equation}
It can also be shown (again, see \cite{dusson2017posteriori}) that the \emph{shifted} second derivative $\mathcal{E}_{\rm eig}''[u^*_{\rm eig}] -\lambda^* \colon X \rightarrow X^*$ satisfies the estimates
\begin{equation}\label{eq:shifted_energy_estimates}
\begin{split}
\exists \alpha_{\mathcal{E}}>0, \forall w \in X\colon  \qquad \alpha_{\mathcal{E}} \Vert w \Vert^2_{X} \leq &\left\langle \big(\mathcal{E}_{\rm eig}''[u^*_{\rm eig}] -\lambda^* \big) w, w\right \rangle_{X \times X^*}\\[0.5em]
\exists \beta_{\mathcal{E}}>0, ~ \forall z, w \in X\colon  \hspace{1.41cm} \quad &\left\langle \big(\mathcal{E}_{\rm eig}''[u^*_{\rm eig}] -\lambda^* \big)z, w\right \rangle_{X \times X^*}\leq \beta_{\mathcal{E}} \Vert z\Vert_{X} \Vert w\Vert_{X}.
\end{split}
\end{equation}


Following the notation used earlier in Section \ref{sec:source}, let us denote by $\{X_{\delta}\}_{\delta>0}$ a sequence of conforming approximation spaces of $X$ that satisfy the approximation property \eqref{eq:approximation_property}. Next, for any~$\delta>0$, we denote by $u^*_{\delta, \rm eig} \in X_\delta$ the solution of the discrete constrained minimisation problem~\eqref{eq:1.5} in $X_\delta$ that satisfies $\big( u^*_{\delta, \rm eig}, u^*_{\rm eig}\big)_{L^2(\Omega)}\geq 0$. As mentioned in the introduction, the existence of such a discrete minimiser (at least for $\delta$ small enough) is guaranteed for the classical choices of finite element, spectral polynomial and spectral Fourier approximation spaces $\{X_\delta\}_{\delta>0}$. It is furthermore straightforward to deduce that each discrete minimiser $u^*_{\delta, \rm eig} \in X_\delta$ satisfies the discrete eigenvalue problem
\begin{equation}\label{eq:Euler-Lagrange_discrete}
\forall v_\delta \in X_\delta \colon \qquad \langle \mathcal{A}_{u^*_{\delta, \rm eig}}  u^*_{\delta, \rm eig} , v_\delta\rangle_{X \times X^*}= \lambda^*_\delta (u^*_{\delta, \rm eig}, v_\delta)_{L^2(\Omega)}.    
\end{equation}
Indeed, Equation \eqref{eq:Euler-Lagrange_discrete} is simply the Euler-Lagrange equation associated with the \emph{discrete} constrained minimisation problem~\eqref{eq:1.6}. As before, let us emphasise that $\lambda^*_{\delta}$ is the smallest eigenvalue of the discrete operator $\mathcal{A}_{u^*_{\delta, \rm eig}} \colon X_\delta \rightarrow (X_\delta)^*$ (again, see\cite{dusson2017posteriori}). 
Finally, we recall the following important result from \cite[Theorem 1]{cances2010numerical}: denoting by $\lambda^*$ and $\lambda^*_\delta$, the lowest eigenvalues  of the infinite-dimensional eigenvalue problem~\eqref{eq:Euler_Lagrange_continuous} and the finite-dimensional eigenvalue problem~\eqref{eq:Euler-Lagrange_discrete} respectively, it holds that (remember that $\lambda^*_\delta \ge  \lambda^*$)
    \begin{align}\label{eq:eigenvalue_convergence}
     \lambda^*_\delta - \lambda^*  \lsim \Vert u^*_{\rm eig}- u^*_{\delta, \rm eig}\Vert^2_{X} + \Vert u^*_{\rm eig}- u^*_{\delta, \rm eig}\Vert_{L^2(\Omega)}.
\end{align}

\vspace{2mm}

As in Section~\ref{sec:source}, we now impose some assumptions on the present abstract Galerkin framework. These assumptions are simply analogues of the assumptions imposed in Section~\ref{sec:source}, adapted to the eigenvalue problem setting.  
\vspace{1mm}

\begin{description}
    \item[Assumption B.1 (Explicit Convergence Rates)]
    \end{description}
    We assume that there exists $r_{\rm eig} \geq 2 $ such that for all $\delta>0$ sufficiently small it holds that
    \begin{align}\label{eq:assum_eig_1a}
    \Vert u_{\rm eig}^* - u^*_{\delta, \rm eig} \Vert_{L^2(\Omega)} + \delta\Vert u_{\rm eig}^* - u^*_{\delta, \rm eig} \Vert_{X} \lsim \delta \underset{v_\delta \in X_\delta}{\inf} \Vert u_{\rm eig}^* - v_{\delta} \Vert_{X} \lsim \delta^{r_{\rm eig}} \Vert u_{\rm eig}^*\Vert_{H^{r_{\rm eig}}(\Omega)}
    \end{align}
    Additionally, we assume that the $(\cdot, \cdot)_X$-orthogonal projection operator $\Pi_\delta^X \colon X \rightarrow X_\delta$ satisfies for all $\delta>0$ sufficiently small and all $v \in H^2(\Omega)$ the estimate
    \begin{align}\label{eq:assum_eig_1b}
    \Vert v -\Pi_\delta^X v \Vert_{L^2(\Omega)} \leq \delta \Vert v -\Pi_\delta^X v \Vert_{X} \lsim \delta^{2} \Vert v\Vert_{H^{2}(\Omega)}.
    \end{align}

    \vspace{1mm}

\begin{description}
    \item[Assumption B.2 (Uniform Boundedness of Discrete Minimisers)]
    \end{description}
    We assume that the sequence of discrete constrained minimisers $\{u^{*}_{\delta, \rm eig}\}_{\delta >0}$ is uniformly bounded in $L^{\infty}(\Omega)$, i.e., there exists a $C>0$ such that for all $\delta >0$ it holds that
    \begin{align}\label{eq:assum_eig_3}
\Vert u^{*}_{\delta, \rm eig} \Vert_{L^{\infty}(\Omega)} \leq C.
    \end{align}
    
\vspace{2mm}

Let us remark here that -- similar to \textbf{Assumption A.2} in Section \ref{sec:source} -- \textbf{Assumption B.2} can be replaced with alternative assumptions based on the existence of a suitable projection operator and inverse inequalities (see Remark \ref{rem:inverse}).
\vspace{2mm}

Analogously to Section \ref{sec:source}, our stronger superconvergence result for Galerkin approximations to the constrained minimisation problem~\eqref{eq:1.2} will require a further assumption (cf. \textbf{Assumption A.3}) in addition to \textbf{Assumptions B.1} and \textbf{B.2}.

\begin{description}
    \item[Assumption B.3 (A Convergence Result for Generic Potentials)]
    \end{description}
   Recalling that $r_v >d/2$ is the Sobolev regularity index of the effective potential $V$, let $\widetilde{W} \in H^{r_v}(\Omega)$ be a generic potential and for any $\delta>0$, let $\Pi^X_\delta \colon X \rightarrow X_\delta$ denote the $(\cdot, \cdot)_X$-orthogonal projection operator onto $X_\delta$. We assume that there exists a real number $q^*_{\rm eig}\in (0, 2]$ such that for all $\delta>0$ sufficiently small, all $\varphi\in H^2(\Omega)\cap X$, and all $\chi \in X$ it holds that
    \begin{align}\label{eq:assum_4}
    \int_\Omega  \widetilde{W}(\bold{x}) \varphi({\bf x}) [\chi- \Pi^X_\delta \chi]({\bf x})\; d{\bf{x}} &\lsim \delta^{q^*_{\rm eig}+1}\Vert \varphi \Vert_{H^2(\Omega)} \Vert  \chi- \Pi^X_\delta \chi\Vert_{X},
    \end{align}
    and for all $\delta>0$ sufficiently small and all discrete constrained minimisers  $\{u^{*}_{\delta, \rm eig}\}_{\delta >0}$ it holds that
    \begin{align}
    \int_\Omega u^*_{\rm eig}({\bf x}) [u^*_{\rm eig}- u^*_{\delta, \rm eig}]({\bf x})\; d{\bf{x}} &\lsim \delta^{q^*_{\rm eig}+1}\Vert u^*_{\rm eig} \Vert_{H^2(\Omega)} \Vert  u^*_{\rm eig}- u^*_{\delta, \rm eig}\Vert_{X}. \label{eq:assum_4new}
\end{align}

\vspace{2mm}
We remark that \textbf{Assumption B.3} above is similar to \textbf{Assumption A.3} from the previous Section~\ref{sec:source}. The main difference between these two assumptions is (i) the auxiliary potential function $\widetilde{V}_{\delta}, ~\delta >0$ appearing in \textbf{Assumption A.3} is replaced with a $\delta$-independent generic potential $\widetilde{W}\in H^{r_v}(\Omega)$ in \textbf{Assumption B.3}, and (ii) we require an analogous convergence estimate to hold when the $(\cdot, \cdot)_X$-best approximation difference  $\chi- \Pi^X_\delta \chi$ is replaced with the Galerkin approximation difference $u^*_{\rm eig}- u^*_{\delta, \rm eig}$ and $\widetilde{W}\equiv 1$.

\vspace{2mm}
Equipped with \textbf{Assumptions B.1, B.2} and possibly  \textbf{B.3}, we are now ready to state and prove our abstract superconvergence result.

\begin{lemma}\label{lem:combine_eig_2}
   Assume the framework adopted in the current section, let \textbf{Assumptions B.1} and \textbf{B.2} hold, and for every $\delta>0$ let $\Pi^X_\delta \colon X \rightarrow X_\delta$ denote the $(\cdot, \cdot)_X$-orthogonal projection operator onto $X_\delta$. Then we have the error estimate
    \begin{align}\label{eq:eig_lower_rates}
      \Vert \Pi_\delta^X u^*_{\rm eig}-u_{\delta, \rm eig}^{*}\Vert_{X} &\lsim   \delta\Vert \Pi^X_{\delta} u^*_{\rm eig}-u^*_{\rm eig}\Vert_{X}.
    \end{align}
    If, in addition, \textbf{Assumption B.3} also holds, then we have the improved error estimate
    \begin{align}\label{eq:eig_higher_rates}
      \Vert \Pi_\delta^X u^*_{\rm eig}-u_{\delta, \rm eig}^{*}\Vert_{L^2(\Omega)} + \delta^{\frac{q^*_{\rm eig}}{2}}\Vert \Pi_\delta^X u^*_{\rm eig}-u_{\delta, \rm eig}^{*}\Vert_{X} &\lsim   \delta^{q^*_{\rm eig}+1}\Vert \Pi_{\delta}^X u^*_{\rm eig}-u_{\rm eig}^{*}\Vert_{X},
    \end{align}
    where $q^*_{\rm eig}\in (0, 2]$ is the constant introduced in \textbf{Assumption B.3}.
\end{lemma}
\begin{proof}
The first portion of the proof involves a straight-forward albeit tedious calculation. For the sake of brevity, we refrain from presenting all of the details but the interested reader can find this calculation presented in detail in \cite[Theorem 5.1.2]{wang2023posteriori} for the specific case of spectral Fourier approximation spaces.

Let $\delta>0$ be fixed. Using the expression for the second derivative $\mathcal{E}_{\rm eig}''[u^*_{\rm eig}] \colon X \rightarrow X^*$ offered by Equation \eqref{eq:energy_derivative}, we can verify that for any $ v_\delta\in X_\delta $ it holds that
\begin{equation}\label{eq:discret_eigen}
\begin{aligned}
 \left \langle \big(\mathcal{E}_{\rm eig}''[u^*_{\rm eig}]-\lambda^*\big)(u^*_{\delta, \rm eig} - u^*_{\rm eig}), v_\delta\right\rangle_{X^* \times X}&=(\lambda^*_\delta-\lambda^*)\int_\Omega u^*_{\delta, \rm eig}(\bold{x})v_\delta(\bold{x}) \; d\bold{x}\\[0.25em]
 &-\int_\Omega \big(u^*_{\delta, \rm eig}(\bold{x})-u^*_{\rm eig}(\bold{x})\big)^2\big(u^*_{\delta, \rm eig}(\bold{x})+2u^*_{\rm eig}(\bold{x})\big)v_\delta(\bold{x})\; d \bold{x}.
\end{aligned}
\end{equation}
Making use of the above Equation \eqref{eq:discret_eigen} and appealing again to the expression for the second derivative $\mathcal{E}_{\rm eig}''[u^*_{\rm eig}]-\lambda^*$ offered by Equation \eqref{eq:energy_derivative} yields that for all $ v_\delta\in X_\delta $ it also holds that
\begin{equation}\label{eq:polish_eigen}
\begin{aligned}
\left \langle \big(\mathcal{E}_{\rm eig}''[u^*_{\rm eig}]-\lambda^*\big)(u^*_{\delta, \rm eig} - \Pi^X_\delta u^*_{\rm eig}), v_\delta\right\rangle_{X^* \times X}&=\left \langle \big(\mathcal{E}_{\rm eig}''[u^*_{\rm eig}]-\lambda^*\big)(u^*_{\delta, \rm eig} - u^*_{\rm eig}), v_\delta \right\rangle_{X^* \times X}\\[0.25em]
&+\left \langle \big(\mathcal{E}_{\rm eig}''[u^*_{\rm eig}]-\lambda^*\big)(u^*_{\rm eig} - \Pi^X_\delta u^*_{\rm eig}), v_\delta\right\rangle_{X^* \times X}\\[0.25em]
&=(\lambda_\delta^*-\lambda^*)\int_\Omega u^*_{\delta, \rm eig}(\bold{x})v_\delta(\bold{x})\; d\bold{x}\\
&-\int_\Omega \big(u^*_{\delta, \rm eig}(\bold{x})-u^*_{\rm eig}(\bold{x})\big)^2\big(u^*_{\delta, \rm eig}(\bold{x})+2u^*_{\rm eig}(\bold{x})\big)v_\delta(\bold{x})\; d \bold{x}\\[0.25em]
&+\int_\Omega \big(V(\bold{x})-\lambda^*+3(u^*_{\rm eig}(\bold{x}))^2\big)\big(u^*_{\rm eig}(\bold{x}) - \Pi^X_\delta u^*_{\rm eig}(\bold{x})\big)v_\delta(\bold{x})\; d \bold{x}.
\end{aligned}
\end{equation}
In particular, setting $ v_\delta= u^*_{\delta, \rm eig} - \Pi^X_\delta u^*_{\rm eig} $ in Equation \eqref{eq:polish_eigen}, introducing the auxiliary potentials
\begin{align}\label{eq:polish_eigen_00}
  V^{(1)}_{\delta, \rm aux}:=  V-\lambda^*+(u^*_{\rm eig})^2 + (u^*_{\delta, \rm eig})^2 + u^*_{\rm eig}u^*_{\delta, \rm eig}, \quad \text{and} \quad V^{(2)}_{\rm aux}:=  V-\lambda^*+3(u^*_{\rm eig})^2,
\end{align}
and performing some simple algebra allows us to deduce that
\begin{equation}\label{eq:polish_eigen_1}
\begin{aligned}
\left \langle \big(\mathcal{E}_{\rm eig}''[u^*_{\rm eig}]-\lambda^*\big)(u^*_{\delta, \rm eig} - \Pi^X_\delta u^*_{\rm eig}), \big(u^*_{\delta, \rm eig} - \Pi^X_\delta u^*_{\rm eig}\big)\right\rangle_{X^* \times X}=(\lambda^*_\delta-\lambda^*)\int_\Omega u^*_{\delta, \rm eig}(\bold{x})\big(u^*_{\delta, \rm eig} - \Pi^X_\delta u^*_{\rm eig}\big)\; d\bold{x}\\[0.5em]
+\int_\Omega V^{(1)}_{\delta, \rm aux}(\bold{x})\big(u^*_{\rm eig}(\bold{x}) -u^*_{\delta, \rm eig}(\bold{x})\big)\big(u^*_{\delta, \rm eig} - \Pi^X_\delta u^*_{\rm eig}\big)\; d \bold{x} +\int_\Omega V^{(2)}_{\rm aux}(\bold{x})\big(u^*_{\delta, \rm eig}(\bold{x}) - \Pi^X_\delta u^*_{\rm eig}(\bold{x})\big)^2\; d \bold{x}.
\end{aligned}
\end{equation}
Appealing now to the eigenvalue bound \eqref{eq:eigenvalue_convergence}, and recalling the normalisation $\Vert u^*_{\delta, \rm eig}\Vert_{L^{2}(\Omega)}=1$ allows us to deduce that
\begin{equation}\label{eq:polish_eigen_1b}
\begin{aligned}
\Big \langle \big(\mathcal{E}_{\rm eig}''[u^*_{\rm eig}]-\lambda^*\big)(u^*_{\delta, \rm eig} - \Pi^X_\delta u^*_{\rm eig}), &\big(u^*_{\delta, \rm eig} - \Pi^X_\delta u^*_{\rm eig}\big)\Big\rangle_{X^* \times X} \\[0.75em]
\lsim  &\big(\Vert u^*_{\delta, \rm eig}-u^*_{\rm eig}\Vert_{X}^2+\Vert u^*_{\delta, \rm eig}-u^*_{\rm eig}\Vert_{L^2(\Omega)}\big)\Vert u^*_{\delta, \rm eig} - \Pi^X_\delta u^*_{\rm eig} \Vert_{L^2(\Omega)}\\[0.25em]
+&\Vert V^{(1)}_{\delta, \rm aux}\Vert_{L^{\infty}(\Omega)} \Vert u^*_{\rm eig}-u^*_{\delta, \rm eig}\Vert_{L^{2}(\Omega)} \Vert u^*_{\delta, \rm eig} - \Pi^X_\delta u^*_{\rm eig}\Vert_{L^{2}(\Omega)}\\[0.25em]
+& \Vert V^{(2)}_{\rm aux}\Vert_{L^{\infty}(\Omega)} \Vert u^*_{\delta, \rm eig} - \Pi^X_\delta u^*_{\rm eig}\Vert^2_{L^{2}(\Omega)}. 
\end{aligned}
\end{equation}

Notice now that the effective potential $V \in H^{r_v}(\Omega)$ for some $r_v > d/2$ by hypothesis and the exact constrained minimiser $u^*_{\rm eig}\in H^2(\Omega)$ (recall the elliptic regularity estimate \eqref{eq:elliptic_regularity}). Additionally, it follows from \textbf{Assumption B.2} that the discrete constrained minimisers $\{u^*_{\rm eig, \delta}\}_{\delta >0}$ are uniformly bounded in $L^{\infty}(\Omega)$. Appealing therefore to the Sobolev embedding $H^s(\Omega)  \hookrightarrow L^{\infty}(\Omega)$ for any $s > d/2$, we deduce that the auxiliary potentials $\{V^{(1)}_{\delta, \rm aux}\}_{\delta >0}$ and $V^{(2)}_{\rm aux}$ are also uniformly bounded in $L^{\infty}(\Omega)$. Inequality \eqref{eq:polish_eigen_1b} therefore simplifies to the bound
\begin{equation}\label{eq:polish_eigen_2}
\begin{aligned}
&\Big \langle \big(\mathcal{E}_{\rm eig}''[u^*_{\rm eig}]-\lambda^*\big)(u^*_{\delta, \rm eig} - \Pi^X_\delta u^*_{\rm eig}), \big(u^*_{\delta, \rm eig} - \Pi^X_\delta u^*_{\rm eig}\big)\Big\rangle_{X^* \times X} \\[0.75em]
\lsim  &\big(\Vert u^*_{\delta, \rm eig}-u^*_{\rm eig}\Vert_{X}^2+\Vert u^*_{\delta, \rm eig}-u^*_{\rm eig}\Vert_{L^2(\Omega)}\big)\Vert u^*_{\delta, \rm eig} - \Pi^X_\delta u^*_{\rm eig} \Vert_{L^2(\Omega)}+ \Vert u^*_{\delta, \rm eig} - \Pi^X_\delta u^*_{\rm eig}\Vert^2_{L^{2}(\Omega)}. 
\end{aligned}
\end{equation}

Combining now the above Inequality \eqref{eq:polish_eigen_2} with the coercivity property \eqref{eq:shifted_energy_estimates} of the shifted second derivative of the energy $\mathcal{E}_{\rm eig}$, we obtain that
 \begin{align}\nonumber
 \Vert u^*_{\delta, \rm eig} - \Pi^X_\delta u^*_{\rm eig} \Vert^2_{X} \lsim  &\big(\Vert u^*_{\delta, \rm eig}-u^*_{\rm eig}\Vert_{X}^2+\Vert u^*_{\delta, \rm eig}-u^*_{\rm eig}\Vert_{L^2(\Omega)}\big)\Vert u^*_{\delta, \rm eig} - \Pi^X_\delta u^*_{\rm eig} \Vert_{L^2(\Omega)}\\ \nonumber
 +& \Vert u^*_{\delta, \rm eig} - \Pi^X_\delta u^*_{\rm eig}\Vert^2_{L^{2}(\Omega)}\\[0.5em]  \label{est:polish_1}
 \lsim  &\big(\Vert u^*_{\delta, \rm eig}-u^*_{\rm eig}\Vert_{X}^2+\Vert \Pi^X_\delta u^*_{\rm eig}- u^*_{\rm eig}\Vert_{L^2(\Omega)}\big)\Vert u^*_{\delta, \rm eig} - \Pi^X_\delta u^*_{\rm eig} \Vert_{L^2(\Omega)}\\
 +& \Vert u^*_{\delta, \rm eig} - \Pi^X_\delta u^*_{\rm eig}\Vert^2_{L^{2}(\Omega)},\nonumber
 \end{align}
where the last step follows from a simple application of the triangle inequality. Making use of the convergence rates \eqref{eq:assum_eig_1a} and \eqref{eq:assum_eig_1b} postulated in \textbf{Assumption B.1} together with another  judicious application of the triangle inequality now yields
\begin{align}\nonumber
\Vert u^*_{\delta, \rm eig} - \Pi^X_\delta u^*_{\rm eig} \Vert^2_{X} &\lsim  \big(\delta\Vert u^*_{\rm eig}\Vert_{H^2(\Omega)}\Vert u^*_{\delta, \rm eig}-u^*_{\rm eig}\Vert_{X}+\delta\Vert \Pi^X_\delta u^*_{\rm eig}- u^*_{\rm eig}\Vert_{X}\big)\delta\Vert u^*_{\delta, \rm eig} -  u^*_{\rm eig} \Vert_{X}\\ \nonumber
&+  \big(\delta\Vert u^*_{\rm eig}\Vert_{H^2(\Omega)}\Vert u^*_{\delta, \rm eig}-u^*_{\rm eig}\Vert_{X}+\delta\Vert \Pi^X_\delta u^*_{\rm eig}- u^*_{\rm eig}\Vert_{X}\big)\delta\Vert u^*_{\rm eig} - \Pi^X_\delta u^*_{\rm eig} \Vert_{X}\\ 
 &+ \delta^2\big(\Vert u^*_{\delta, \rm eig} -  u^*_{\rm eig} \Vert_{X}+\Vert u^*_{\rm eig} - \Pi^X_\delta u^*_{\rm eig} \Vert_{X}\big)^2. \label{eq:split_lemma_eig_1}
 \end{align}

Notice that the convergence rates \eqref{eq:assum_eig_1a} stated in \textbf{Assumption B.1} further implIES that
\begin{align}\label{eq:split_lemma_eig_2}
    \Vert u^*_{\delta, \rm eig} -  u^*_{\rm eig} \Vert_{X} \lsim \underset{v_\delta \in X_\delta}{\inf} \Vert u^*_{\rm eig} - v_\delta \Vert_X \leq \Vert u^*_{\rm eig} - \Pi^X_\delta u^*_{\rm eig} \Vert_{X},
\end{align}
Consequently, combining Estimates \eqref{eq:split_lemma_eig_1} and \eqref{eq:split_lemma_eig_2} and using the fact that $u^*_{\rm eig} \in H^2(\Omega)$ (recall the elliptic regularity result \eqref{eq:elliptic_regularity}) yields the sought-after convergence estimate \eqref{eq:eig_lower_rates}.

To prove the improved convergence estimate \eqref{eq:eig_higher_rates}, we return to Inequality \eqref{est:polish_1}. Our goal now is to derive a sharper upper bound for the term $ \Vert u^*_{\delta, \rm eig} - \Pi^X_\delta u^*_{\rm eig} \Vert_{L^2(\Omega)} $. Similar to the proof of Lemma \ref{lem:combine_2} in Section \ref{sec:source}, we will do so by introducing an adjoint source problem. Unlike the proof of Lemma \ref{lem:combine_2} however, this adjoint problem will be posed in a \emph{subspace} rather than the entire approximation space $X$. 
 
 To this end, we first introduce $\Pi_{u^*_{\rm eig}} \colon L^2(\Omega)  \rightarrow \text{span} \{u^*_{\rm eig}\}$ as the $L^2(\Omega)$-orthogonal projector onto span$\{u^*_{\rm eig}\}$, and we denote by $\Pi_{u^*_{\rm eig}}^{\perp}:= I - \Pi_{u^*_{\rm eig}}$, its complement. We then introduce $X_{u^*_{\rm eig}}^{\perp} := X \cap \text{Ran} \Pi_{u^*_{\rm eig}}^{\perp}$ as the subspace of $X$ consisting of all functions $L^2(\Omega)$-orthogonal to~$u^*_{\rm eig}$. The adjoint source problem that we consider now consists of seeking $\vartheta \in X_{u^*_{\rm eig}}^{\perp}$ that solves the equation 
 \begin{align}\label{eq:adjoint_eigen}
\forall v \in X_{u^*_{\rm eig}}^{\perp}\colon \quad \left\langle \big(\mathcal{E}_{\rm eig}''[u^*_{\rm eig}] -\lambda^* \big)v, \vartheta\right\rangle_{X^* \times X}= \left(u^*_{\delta, \rm eig} -\Pi^X_\delta u^*_{\rm eig}, v\right)_{L^2(\Omega)}.     
 \end{align}
 Since $\mathcal{E}_{\rm eig}''[u^*_{\rm eig}]  \colon X \rightarrow X^*$ is self-adjoint, and we have the coercivity and continuity properties~\eqref{eq:shifted_energy_estimates}, we immediately deduce that the adjoint problem \eqref{eq:adjoint_eigen} has a unique solution $\vartheta \in X_{u^*_{\rm eig}}^{\perp}$. Additionally, in view of the definition of $\mathcal{E}_{\rm eig}''[u^*_{\rm eig}]$ given by Equation \eqref{eq:energy_derivative}, we can deduce from elliptic regularity theory and a bootstrapping argument that $\vartheta \in H^2(\Omega)$ and satisfies the estimate
 \begin{align}\label{eq:adjoint_eigen_estimate}
     \Vert \vartheta \Vert_{H^2(\Omega)} \lsim \Vert u^*_{\delta, \rm eig} -\Pi^X_\delta u^*_{\rm eig}\Vert_{L^2(\Omega)}.
 \end{align}

As usual, the trick is now to plug in a judicious test function $v \in X_{u^*_{\rm eig}}^{\perp}$ in the adjoint equation~\eqref{eq:adjoint_eigen}. Unlike the adjoint problem considered in the proof of Lemma \ref{lem:combine_2} in Section \ref{sec:source} however, the correct choice of test function in the present setting is a bit subtle, and we require some additional tools. To begin with, we introduce the affine space $\mathcal{S}_{\delta} \subset X$ as
\begin{align*}
    \mathcal{S}_{\delta}:= \Pi^X_{\delta} u^*_{\rm eig} + X_{u^*_{\rm eig}}^{\perp}.
\end{align*}
 Next, we introduce the mapping $\mathbb{P}_{\delta}\colon X \rightarrow \mathcal{S}_{\delta}$ onto this affine space~$\mathcal{S}_{\delta}$ as
\begin{align*}
    \forall v \in   X \colon \qquad \mathbb{P}_{\delta} v :=  \Pi^X_\delta u^*_{\rm eig} + \Pi_{u^*_{\rm eig}}^\perp \big(v -\Pi^X_\delta u^*_{\rm eig}\big).
\end{align*}

 Several properties of the mapping $\mathbb{P}_{\delta}$ can now be deduced. To begin with, we note that
 \begin{align*}
    \big(\mathbb{P}_{\delta}u^*_{\delta, \rm eig} - \Pi^X_\delta u^*_{\rm eig}, u^*_{\rm eig}\big)_{L^2(\Omega)}=0 \quad \text{i.e.,} \quad \mathbb{P}_{\delta}u^*_{\delta, \rm eig} - \Pi^X_\delta u^*_{\rm eig} \in X^{\perp}_{u^*_{\rm eig}}.
\end{align*}
Additionally, a direct calculation reveals that
 \begin{align}\nonumber
     \mathbb{P}_{\delta}u^*_{\delta, \rm eig}-  u^*_{\delta, \rm eig}&= \Pi^X_\delta u^*_{\rm eig} + \Pi_{u^*_{\rm eig}}^\perp \big(u^*_{\delta, \rm eig} -\Pi^X_\delta u^*_{\rm eig}\big)-  u^*_{\delta, \rm eig}\\  \label{eq:new_proof_01}
     &=\Pi^X_\delta u^*_{\rm eig} -\Pi_{u^*_{\rm eig}}^\perp \Pi^X_\delta u^*_{\rm eig} +\Pi_{u^*_{\rm eig}}^\perp u^*_{\delta, \rm eig} -u^*_{\delta, \rm eig}\\ \nonumber
     &=\Pi_{u^*_{\rm eig}}\Pi^X_\delta u^*_{\rm eig} - \Pi_{u^*_{\rm eig}} u^*_{\delta, \rm eig}.
 \end{align}
 Equation \eqref{eq:new_proof_01} implies in particular that
 \begin{align}\label{eq:new_proof_2}
 \mathbb{P}_{\delta}u^*_{\delta, \rm eig}-u^*_{\delta, \rm eig} = \overline{c}\; u^*_{\rm eig} \quad \text{with} \quad
 \overline{c}= \big(\Pi_\delta^X u^*_{\rm eig} -u^*_{\delta, \rm eig}, u^*_{\rm eig}\big)_{L^2(\Omega)}.
 \end{align}
Using the above Equation \eqref{eq:new_proof_2}, we can now deduce that
\begin{align}\nonumber
     \Vert u^*_{\delta, \rm eig} - \Pi^X_\delta u^*_{\rm eig} \Vert^2_{L^2(\Omega)}&= \big(u^*_{\delta, \rm eig} - \Pi^X_\delta u^*_{\rm eig}, \mathbb{P}_{\delta}u^*_{\delta, \rm eig} - \Pi^X_\delta u^*_{\rm eig}\big)_{L^2(\Omega)}\\ \nonumber
     &+ \big(u^*_{\delta, \rm eig} - \Pi^X_\delta u^*_{\rm eig}, u^*_{\delta, \rm eig}-\mathbb{P}_{\delta}u^*_{\delta, \rm eig} \big)_{L^2(\Omega)}\\ \nonumber
     &=\big(u^*_{\delta, \rm eig} - \Pi^X_\delta u^*_{\rm eig}, \mathbb{P}_{\delta}u^*_{\delta, \rm eig} - \Pi^X_\delta u^*_{\rm eig}\big)_{L^2(\Omega)}- \overline{c}\big(u^*_{\delta, \rm eig} - \Pi^X_\delta u^*_{\rm eig}, u^*_{\rm eig} \big)_{L^2(\Omega)}\\ \label{eq:revision_hassan}
     &=\big(u^*_{\delta, \rm eig} - \Pi^X_\delta u^*_{\rm eig}, \mathbb{P}_{\delta}u^*_{\delta, \rm eig} - \Pi^X_\delta u^*_{\rm eig}\big)_{L^2(\Omega)}+ \overline{c}^2.
\end{align}
Since, moreover, $\mathbb{P}_{\delta}u^*_{\delta, \rm eig} - \Pi^X_\delta u^*_{\rm eig} \in X^{\perp}_{u^*_{\rm eig}}$, we can plug in the test function $v= \mathbb{P}_{\delta}u^*_{\delta, \rm eig} - \Pi^X_\delta u^*_{\rm eig}$ in the adjoint problem \eqref{eq:adjoint_eigen}, and obtain from Equation \eqref{eq:revision_hassan} that
\begin{align*}
     \Vert u^*_{\delta, \rm eig} - \Pi^X_\delta u^*_{\rm eig} \Vert^2_{L^2(\Omega)}&= \left\langle \big(\mathcal{E}_{\rm eig}''[u^*_{\rm eig}] -\lambda^* \big)\big(\mathbb{P}_{\delta}u^*_{\delta, \rm eig} - \Pi^X_\delta u^*_{\rm eig}\big), \vartheta\right\rangle_{X^* \times X} + \overline{c}^2.
\end{align*}
Making use once again of Equation \eqref{eq:new_proof_2}, we can further deduce that
\begin{align*}
     \Vert u^*_{\delta, \rm eig} - \Pi^X_\delta u^*_{\rm eig} \Vert^2_{L^2(\Omega)}&= \left\langle \big(\mathcal{E}_{\rm eig}''[u^*_{\rm eig}] -\lambda^* \big)\big(u^*_{\delta, \rm eig} - \Pi^X_\delta u^*_{\rm eig}\big), \vartheta\right\rangle_{X^* \times X}\\
     &+\left\langle \big(\mathcal{E}_{\rm eig}''[u^*_{\rm eig}] -\lambda^* \big)\big(\mathbb{P}_{\delta}u^*_{\delta, \rm eig} - u^*_{\delta, \rm eig}\big), \vartheta\right\rangle_{X^* \times X}+ \overline{c}^2\\[0.5em]
     &= \left\langle \big(\mathcal{E}_{\rm eig}''[u^*_{\rm eig}] -\lambda^* \big)\big(u^*_{\delta, \rm eig} - \Pi^X_\delta u^*_{\rm eig}\big), \vartheta\right\rangle_{X^* \times X}\\
     &+\overline{c}\left\langle \big(\mathcal{E}_{\rm eig}''[u^*_{\rm eig}] -\lambda^* \big) u^*_{\rm eig}, \vartheta\right\rangle_{X^* \times X}+ \overline{c}^2.
\end{align*}
Recalling now the definition of the second derivative $\mathcal{E}_{\rm eig}''[u^*_{\rm eig}]$ of the energy functional given by Equation \eqref{eq:energy_derivative} and noting that $\vartheta \in X_{u^*_{\rm eig}}^{\perp}$, we finally obtain
\begin{align}
     \Vert u^*_{\delta, \rm eig} - \Pi^X_\delta u^*_{\rm eig} \Vert^2_{L^2(\Omega)}&= \left\langle \big(\mathcal{E}_{\rm eig}''[u^*_{\rm eig}] -\lambda^* \big)\big(u^*_{\delta, \rm eig} - \Pi^X_\delta u^*_{\rm eig}\big), \vartheta\right\rangle_{X^* \times X}+2\overline{c}\int_{\Omega}(u^*_{\rm eig}(\bold{x}))^3\vartheta(\bold{x})\; d\bold{x} + \overline{c}^2.\label{eq:new_proof_3}
\end{align}

By inserting the projection operator $\Pi_\delta^{X} \colon X \rightarrow X_\delta$ and making use of the earlier Equation \eqref{eq:polish_eigen}, we can rewrite Equation \eqref{eq:new_proof_3} as
\begin{align}\nonumber
     \Vert u^*_{\delta, \rm eig} - \Pi^X_\delta u^*_{\rm eig} \Vert^2_{L^2(\Omega)}&= \left\langle \big(\mathcal{E}_{\rm eig}''[u^*_{\rm eig}] -\lambda^* \big)\big(u^*_{\delta, \rm eig} - \Pi^X_\delta u^*_{\rm eig}\big), \vartheta- \Pi_\delta^{X} \vartheta\right\rangle_{X^* \times X}+2\overline{c}\int_{\Omega}(u^*_{\rm eig}(\bold{x}))^3\vartheta(\bold{x})\; d\bold{x}+ \overline{c}^2\\[0.25em] \nonumber
     &+\left\langle \big(\mathcal{E}_{\rm eig}''[u^*_{\rm eig}] -\lambda^*\big)\big(u^*_{\delta, \rm eig} - \Pi^X_\delta u^*_{\rm eig}\big),\Pi_\delta^{X} \vartheta\right\rangle_{X^* \times X}\\[1em] \nonumber
     &= \underbrace{\left\langle \big(\mathcal{E}_{\rm eig}''[u^*_{\rm eig}] -\lambda^* \big)\big(u^*_{\delta, \rm eig} - \Pi^X_\delta u^*_{\rm eig}\big), \vartheta- \Pi_\delta^{X} \vartheta\right\rangle_{X^* \times X}}_{\rm :=(I)}\\ \label{eq:eig_main_L2}
&+\underbrace{2\overline{c}\int_{\Omega}(u^*_{\rm eig}(\bold{x}))^3\vartheta(\bold{x})\; d\bold{x}+\overline{c}^2}_{\rm := (II)}+\underbrace{(\lambda^*_\delta-\lambda^*)\int_\Omega u^*_{\delta, \rm eig}(\bold{x})\Pi_\delta^{X} \vartheta(\bold{x})\; d\bold{x}}_{:= \rm (III)}\\ \nonumber
&\underbrace{-\int_\Omega \big(u^*_{\delta, \rm eig}(\bold{x})-u^*_{\rm eig}(\bold{x})\big)^2\big(u^*_{\delta, \rm eig}(\bold{x})+2u^*_{\rm eig}(\bold{x})\big)\Pi_\delta^{X} \vartheta(\bold{x})\; d \bold{x}}_{\rm := (IV)}\\ \nonumber
&+\underbrace{\int_\Omega V^{(2)}_{\rm eig}(\bold{x})\big(u^*_{\rm eig}(\bold{x}) - \Pi^X_\delta u^*_{\rm eig}(\bold{x})\big)\Pi_\delta^{X} \vartheta(\bold{x})\; d \bold{x}}_{\rm := (V)}.
\end{align}

We now proceed term-by-term beginning with (I). This first term can be estimated by making use of the continuity property \eqref{eq:shifted_energy_estimates} of the shifted second derivative $\mathcal{E}_{\rm eig}''-\lambda^*$ together with \textbf{Assumption B.1} and Estimate \eqref{eq:adjoint_eigen_estimate} to yield that
\begin{align}\nonumber
   {\rm (I)} \lsim \Vert u^*_{\delta, \rm eig} - \Pi^X_\delta u^*_{\rm eig}\Vert_X \Vert \vartheta- \Pi_\delta^{X} \vartheta\Vert_X &\lsim \delta\Vert u^*_{\delta, \rm eig} - \Pi^X_\delta u^*_{\rm eig}\Vert_X \Vert \vartheta\Vert_{H^2(\Omega)}\\ \label{eq:new_proof_estimate_1}
   &\lsim \delta\Vert u^*_{\delta, \rm eig} - \Pi^X_\delta u^*_{\rm eig}\Vert_X \Vert  u^*_{\delta, \rm eig} - \Pi^X_\delta u^*_{\rm eig}\Vert_{L^2(\Omega)}.
\end{align}

In order to estimate the second term, we make use of the Cauchy-Schwarz inequality, the Sobolev embedding $H^1(\Omega) \hookrightarrow L^6(\Omega)$ and Estimate \eqref{eq:adjoint_eigen_estimate} to first deduce that
\begin{align}\label{eq:new_proof_estimate_2_aux}
    {\rm (II)} \lsim \overline{c}\Vert u^*_{\rm eig}\Vert^3_{L^6(\Omega)}\Vert \vartheta\Vert_{L^2(\Omega)}+ \overline{c}^2 \lsim  \overline{c}\Vert u^*_{\delta, \rm eig} - \Pi^X_\delta u^*_{\rm eig}\Vert_{L^2(\Omega)}+ \overline{c}^2.
\end{align}
It therefore remains to simplify the scalar $\overline{c}$. To this end, we make use of the expression for this constant given by Equation \eqref{eq:new_proof_2} to deduce that
\begin{align}\label{eq:revision_hassan_0}
    \overline{c}=   \big(\Pi_\delta^X u^*_{\rm eig} -u^*_{\delta, \rm eig}, u^*_{\rm eig}\big)_{L^2(\Omega)}= \big(\Pi_\delta^X u^*_{\rm eig} -u^*_{\rm eig}, u^*_{\rm eig}\big)_{L^2(\Omega)} + \big( u^*_{\rm eig} -u^*_{\delta, \rm eig}, u^*_{\rm eig}\big)_{L^2(\Omega)}.
\end{align}
In order to simplify Inequality \eqref{eq:revision_hassan_0}, we now appeal to Estimate \eqref{eq:assum_4} with $\widetilde W =1$ and Estimate \eqref{eq:assum_4new} in \textbf{Assumption B.3} to obtain the existence of a real number $q^*_{\rm eig} \in (0, 2]$ such that for all $\delta$ sufficiently small it holds that
\begin{align}\label{eq:new_proof_estimate_2_aux_2}
    \overline{c} \lsim \delta^{q^*_{\rm eig} +1} \Vert u^*_{\rm eig}\Vert_{H^2(\Omega)} \Vert u^*_{\rm eig}- \Pi_\delta^X u^*_{\rm eig}\Vert_X.
\end{align}

Since we additionally possess the trivial bound $ \overline{c}  \leq \Vert u^*_{\delta, \rm eig} - \Pi^X_\delta u^*_{\rm eig}\Vert_{L^2(\Omega)}$ (cf. Equation \eqref{eq:new_proof_2}), Estimates \eqref{eq:new_proof_estimate_2_aux} and \eqref{eq:new_proof_estimate_2_aux_2} together yield that
\begin{align}\label{eq:new_proof_estimate_2}
    {\rm (II)} &\lsim \delta^{q^*_{\rm eig} +1} \Vert u^*_{\rm eig}- \Pi_\delta^X u^*_{\rm eig}\Vert_X\Vert u^*_{\delta, \rm eig} - \Pi^X_\delta u^*_{\rm eig}\Vert_{L^2(\Omega)}.
\end{align}
Here, we have used, as usual, the fact that $u^*_{\rm eig} \in H^2(\Omega)$ (recall the elliptic regularity result \eqref{eq:elliptic_regularity}).

We next turn to the term (III) appearing in Equation \eqref{eq:eig_main_L2}. Let us first observe that since the solution $\vartheta$ to the adjoint problem \eqref{eq:adjoint_eigen} is an element of $X^{\perp}_{u^*}$ by construction, we can write
\begin{align*}
   {\rm (III)}=(\lambda_\delta^*-\lambda^*)\int_\Omega u^*_{\delta, \rm eig}(\bold{x})\Pi_\delta^{X} \vartheta(\bold{x})\; d\bold{x}&= (\lambda_\delta^*-\lambda^*)\int_\Omega \big(u^*_{\delta, \rm eig}(\bold{x})-u^*_{\rm eig}(\bold{x})\big)\Pi_\delta^{X} \vartheta(\bold{x})\; d\bold{x}\\
    &+ (\lambda^*_\delta-\lambda^*)\int_\Omega u^*_{\rm eig}(\bold{x})\Pi_\delta^{X} \vartheta(\bold{x})\; d\bold{x}\\
    &=(\lambda^*_\delta-\lambda^*)\int_\Omega \big(u^*_{\delta, \rm eig}(\bold{x})-u^*_{\rm eig}(\bold{x})\big)\Pi_\delta^{X} \vartheta(\bold{x})\; d\bold{x}\\
    &+ (\lambda^*_\delta-\lambda^*)\int_\Omega u^*_{\rm eig}(\bold{x})\big(\Pi_\delta^{X} \vartheta(\bold{x})-\vartheta(\bold{x})\big)\; d\bold{x}.
\end{align*}
Notice now that \textbf{Assumption B.1} implies that
\begin{align*}
\Vert u^*_{\delta, \rm eig}-u^*_{\rm eig}\Vert_{L^2(\Omega)} \lsim \delta^{2}\Vert u^*_{\rm eig}\Vert_{H^2(\Omega)} \quad \text{and} \quad \Vert \Pi_\delta^{X} \vartheta-\vartheta\Vert_{L^2(\Omega)} \lsim \delta^{2}\Vert \vartheta\Vert_{H^2(\Omega)}.
\end{align*}
Consequently, recalling the eigenvalue bound \eqref{eq:eigenvalue_convergence}, applying the Cauchy-Schwarz inequality, and making use once again of Estimate \eqref{eq:adjoint_eigen_estimate}, we obtain that
\begin{align}\nonumber
   {\rm (III)}&\lsim \delta^{2}\big(\Vert u^*_{\delta, \rm eig}-u^*_{\rm eig}\Vert_{X}^2+\Vert u^*_{\delta, \rm eig}-u^*_{\rm eig}\Vert_{L^2(\Omega)}\big)\left(\Vert u^*_{\rm eig}\Vert_{H^2(\Omega)}\Vert \Pi_\delta^X \vartheta\Vert_{L^2(\Omega)} + \Vert u^*_{\rm eig}\Vert_{L^2(\Omega)}\Vert \vartheta \Vert_{H^2(\Omega)}\right)\\[0.25em] \nonumber
   &\lsim \delta^{2}\big(\Vert u^*_{\delta, \rm eig}-u^*_{\rm eig}\Vert_{X}^2+\Vert u^*_{\delta, \rm eig}-u^*_{\rm eig}\Vert_{L^2(\Omega)}\big)\left(\Vert \vartheta\Vert_{H^1(\Omega)} + \Vert \vartheta \Vert_{H^2(\Omega)}\right)\\[0.25em]
   &\lsim \delta^{2}\big(\Vert u^*_{\delta, \rm eig}-u^*_{\rm eig}\Vert_{X}^2+\Vert u^*_{\delta, \rm eig}-u^*_{\rm eig}\Vert_{L^2(\Omega)}\big)\Vert u^*_{\delta, \rm eig} - \Pi^X_\delta u^*_{\rm eig}\Vert_{L^2(\Omega)}.  \label{eq:new_proof_estimate_3}
\end{align}

\medskip

In order to estimate the term (IV) appearing in Equation \eqref{eq:eig_main_L2}, we notice that we can write
\begin{align*}
    {\rm (IV)} &= \underbrace{-\int_\Omega \big(u^*_{\delta, \rm eig}(\bold{x})-u^*_{\rm eig}(\bold{x})\big)^2\big(u^*_{\delta, \rm eig}(\bold{x})+2u^*_{\rm eig}(\bold{x})\big)\vartheta(\bold{x})\; d \bold{x}}_{:= \rm (IVA)} \\
    &+\underbrace{\int_\Omega \big(u^*_{\delta, \rm eig}(\bold{x})-u^*_{\rm eig}(\bold{x})\big)^2\big(u^*_{\delta, \rm eig}(\bold{x})+2u^*_{\rm eig}(\bold{x})\big)\big(\vartheta-\Pi_\delta^{X} \vartheta\big)(\bold{x})\; d \bold{x}}_{:= \rm (IVB)}.
\end{align*}
Turning first to the term (IVA), we use Hölder's inequality to deduce that
\begin{align*}
    {\rm (IVA)} \leq \Vert u^*_{\delta, \rm eig}-u^*_{\rm eig}\Vert_{L^2(\Omega)}^2 \Vert u^*_{\delta, \rm eig}+2u^*_{\rm eig}\Vert_{L^\infty(\Omega)}\Vert \vartheta\Vert_{L^{\infty}(\Omega)} &\lsim \Vert u^*_{\delta, \rm eig}-u^*_{\rm eig}\Vert_{L^2(\Omega)}^2 \Vert \vartheta\Vert_{H^{2}(\Omega)},
\end{align*}
where the second step follows from the uniform $L^{\infty}(\Omega)$-boundedness of the discrete minimisers $\{u^*_{\delta, \rm eig}\}_{\delta >0}$ yielded by \textbf{Assumption B.2} and the fact that $u^*_{\rm eig}$ is an element of $H^2(\Omega)$ (recall the regularity result \eqref{eq:elliptic_regularity}) and hence also of $L^\infty(\Omega)$. Using now the adjoint estimate \eqref{eq:adjoint_eigen_estimate} together with the convergence rates stated in \textbf{Assumption B.1}, we deduce that
\begin{align}\label{eq:new_proof_estimate_4a}
    {\rm (IVA)}\lsim \Vert u^*_{\delta, \rm eig}-u^*_{\rm eig}\Vert_{L^2(\Omega)}^2 \Vert \vartheta\Vert_{H^{2}(\Omega)} \lsim \delta^2 \Vert u^*_{\delta, \rm eig}-u^*_{\rm eig}\Vert_{L^2(\Omega)}\Vert u^*_{\delta, \rm eig} - \Pi^X_\delta u^*_{\rm eig}\Vert_{L^2(\Omega)}.
\end{align}

In order to estimate the term (IVB), we use the Cauchy-Schwarz inequality together with the uniform $L^{\infty}(\Omega)$-boundedness implied by \textbf{Assumption B.2} to deduce that 
\begin{align*}\nonumber
    {\rm (IVB)} &\leq \Vert u^*_{\delta, \rm eig}-u^*_{\rm eig}\Vert_{L^2(\Omega)} \big \Vert \big(u^*_{\delta, \rm eig}-u^*_{\rm eig}\big)\big(u^*_{\delta, \rm eig}+2u^*_{\rm eig}\big)\big(\vartheta-\Pi_\delta^{X} \vartheta\big)\big\Vert_{L^2(\Omega)}\\ \nonumber
    &\leq \Vert u^*_{\delta, \rm eig}-u^*_{\rm eig}\Vert_{L^2(\Omega)} \big\Vert \big(u^*_{\delta, \rm eig}-u^*_{\rm eig}\big)\big(u^*_{\delta, \rm eig}+2u^*_{\rm eig}\big)\big \Vert_{L^{\infty}(\Omega)} \Vert \vartheta-\Pi_\delta^{X} \vartheta\Vert_{L^2(\Omega)}\\ 
    &\lsim \Vert u^*_{\delta, \rm eig}-u^*_{\rm eig}\Vert_{L^2(\Omega)} \Vert \vartheta-\Pi_\delta^{X} \vartheta\Vert_{L^2(\Omega)}.
\end{align*}
Appealing again to the convergence rates yielded by \textbf{Assumption B.1} and the adjoint estimate~\eqref{eq:adjoint_eigen_estimate}, we obtain that
\begin{align}\label{eq:new_proof_estimate_4b}
    {\rm (IVB)} \lsim \Vert u^*_{\delta, \rm eig}-u^*_{\rm eig}\Vert_{L^2(\Omega)} \Vert \vartheta-\Pi_\delta^{X} \vartheta\Vert_{L^2(\Omega)} \lsim \delta^2\Vert u^*_{\delta, \rm eig}-u^*_{\rm eig}\Vert_{L^2(\Omega)}\Vert u^*_{\delta, \rm eig} - \Pi^X_\delta u^*_{\rm eig}\Vert_{L^2(\Omega)}.
\end{align}

Combining Estimates \eqref{eq:new_proof_estimate_4a} and \eqref{eq:new_proof_estimate_4b}, we deduce that
\begin{align}\label{eq:new_proof_estimate_4}
    {\rm (IV)}\lsim \delta^2 \Vert u^*_{\delta, \rm eig}-u^*_{\rm eig}\Vert_{L^2(\Omega)}\Vert u^*_{\delta, \rm eig} - \Pi^X_\delta u^*_{\rm eig}\Vert_{L^2(\Omega)}.
\end{align}

We finally turn our attention to the term (V) appearing in Equation \eqref{eq:eig_main_L2}. Following the same strategy used to bound the term (IV), we first note that we can write
\begin{align*}
    {\rm (V)}= \underbrace{\int_\Omega V^{(2)}_{\rm eig}(\bold{x})\big(u^*_{\rm eig}(\bold{x}) - \Pi^X_\delta u^*_{\rm eig}(\bold{x})\big)\vartheta(\bold{x})\; d \bold{x}}_{:= \rm (VA)}+ \underbrace{\int_\Omega V^{(2)}_{\rm eig}(\bold{x})\big(u^*_{\rm eig}(\bold{x}) - \Pi^X_\delta u^*_{\rm eig}(\bold{x})\big)\big(\Pi_\delta^X\vartheta-\vartheta\big)\; d \bold{x}}_{:= \rm (VB)}.
\end{align*}
To bound the first term (VA) on the right hand-side, we simply make use of \textbf{Assumption B.3} with generic potential $\widetilde{W}=  V^{(2)}_{\rm eig}$, which yields the existence of a real number $q^*_{\rm eig} \in (0, 2]$ such that for all $\delta>0$ sufficiently small it holds that
\begin{align}\label{eq:new_proof_estimate_5a}
    {\rm (VA)} \lsim  \delta^{q^*_{\rm eig}+1} \Vert  u^*_{\rm eig}- \Pi^X_\delta u^*_{\rm eig}\Vert_{X}\Vert \vartheta \Vert_{H^2(\Omega)}
  \lsim  \delta^{q^*_{\rm eig}+1} \Vert  u^*_{\rm eig}- \Pi^X_\delta u^*_{\rm eig}\Vert_{X}\Vert u^*_{\delta, \rm eig} - \Pi^X_\delta u^*_{\rm eig}\Vert_{L^2(\Omega)},
\end{align}
where the second step follows from the adjoint estimate \eqref{eq:adjoint_eigen_estimate}.

To bound the second term (VB) on the right hand-side, we make use of Hölder's inequality, the convergence rates stated in \textbf{Assumption B.1}, and the adjoint estimate \eqref{eq:adjoint_eigen_estimate} to deduce that
\begin{align}\nonumber
    {\rm (VB)} &\leq  \Vert V^{(2)}_{\rm eig}\Vert_{L^{\infty}(\Omega)} \Vert u^*_{\rm eig} -\Pi^X_\delta u^*_{\rm eig}\Vert_{L^2(\Omega)}\Vert \Pi_\delta^X\vartheta-\vartheta\Vert_{L^2(\Omega)}\\ \label{eq:new_proof_estimate_5b}
    &\lsim \delta^3 \Vert V^{(2)}_{\rm eig}\Vert_{L^{\infty}(\Omega)} \Vert u^*_{\rm eig} -\Pi^X_\delta u^*_{\rm eig}\Vert_{X}\Vert u^*_{\delta, \rm eig} - \Pi^X_\delta u^*_{\rm eig}\Vert_{L^2(\Omega)}.
\end{align}

Combining Estimates \eqref{eq:new_proof_estimate_5a} and \eqref{eq:new_proof_estimate_5b} and making use of the fact that the auxiliary potential $V^{(2)}_{\rm eig} \in L^{\infty}(\Omega)$ (see Equation \eqref{eq:polish_eigen_00}) while $q^*_{\rm eig} \in (0, 2]$ by assumption, we conclude that 
\begin{align}\label{eq:new_proof_estimate_5}
    {\rm (V)} \lsim  \delta^{q^*_{\rm eig}+1} \Vert  u^*_{\rm eig}- \Pi^X_\delta u^*_{\rm eig}\Vert_{X}\Vert u^*_{\delta, \rm eig} - \Pi^X_\delta u^*_{\rm eig}\Vert_{L^2(\Omega)}.
\end{align}

 In view of the estimates \eqref{eq:new_proof_estimate_1}, \eqref{eq:new_proof_estimate_2}, \eqref{eq:new_proof_estimate_3}, \eqref{eq:new_proof_estimate_4}, and \eqref{eq:new_proof_estimate_5} as well as the earlier Equation~\eqref{eq:eig_main_L2}, we deduce that for all $\delta>0$ sufficiently small, we have
\begin{align}\nonumber
    \Vert u^*_{\delta, \rm eig} - \Pi^X_\delta u^*_{\rm eig} \Vert_{L^2(\Omega)} &\lsim \delta \Vert u^*_{\delta, \rm eig} - \Pi^X_\delta u^*_{\rm eig}\Vert_X + \delta^{q_{\rm eig}^*+1}\Vert u^*_{\rm eig} - \Pi^X_\delta u^*_{\rm eig}\Vert_X\\ \label{new_proof_estimate_6}
    &+ \delta^{2}\big(\Vert u^*_{\delta, \rm eig}-u^*_{\rm eig}\Vert_{X}^2+\Vert u^*_{\delta, \rm eig}-u^*_{\rm eig}\Vert_{L^2(\Omega)}\big) + \delta^2 \Vert u^*_{\delta, \rm eig}-u^*_{\rm eig}\Vert_{L^2(\Omega)}.
\end{align}
Appealing now to the convergence rates \eqref{eq:assum_eig_1b} postulated in \textbf{Assumption B.1} and using once again the fact that $q^*_{\rm eig} \in (0, 2]$ by assumption, we can further simplify the above Estimate \eqref{new_proof_estimate_6} to
\begin{align}\nonumber
    \Vert u^*_{\delta, \rm eig} - \Pi^X_\delta u^*_{\rm eig} \Vert_{L^2(\Omega)} &\lsim \delta\Vert u^*_{\delta, \rm eig} - \Pi^X_\delta u^*_{\rm eig}\Vert_X +  \delta^{q_{\rm eig}^*+1}\Vert u^*_{\rm eig} - \Pi^X_\delta u^*_{\rm eig}\Vert_X\\  \nonumber
    &+\delta^{2}\big(\delta\Vert u^*_{\rm eig}-\Pi^X_{\delta} u^*_{\rm eig}\Vert_{X}\Vert u^*_{\rm src}\Vert_{H^2(\Omega)}+\delta\Vert u^*_{\rm eig}-\Pi^X_{\delta}u^*_{\rm eig}\Vert_{X}\big) +\delta^{3} \Vert  u^*_{\rm eig}- \Pi^X_\delta u^*_{\rm eig}\Vert_{X}\\ \label{new_proof_estimate_7}
    &\lsim \delta\Vert u^*_{\delta, \rm eig} - \Pi^X_\delta u^*_{\rm eig}\Vert_X +  \delta^{q_{\rm eig}^*+1}\Vert u^*_{\rm eig} - \Pi^X_\delta u^*_{\rm eig}\Vert_X.
\end{align}
Here, the last step uses, as usual, the fact that $u^*_{\rm eig}\in H^2(\Omega)$. 

The next step, in principle, is to make use of the above Estimate \eqref{new_proof_estimate_7} in the earlier Estimate \eqref{est:polish_1}. Prior to doing so however, it is useful to first use the convergence rates \eqref{eq:assum_eig_1a}-\eqref{eq:assum_eig_1b} postulated in \textbf{Assumption B.1} to bound Estimate \eqref{new_proof_estimate_7} in the following manner:
\begin{align}\nonumber
    \Vert u^*_{\delta, \rm eig} - \Pi^X_\delta u^*_{\rm eig} \Vert^2_{X} &\lsim  \big(\Vert u^*_{\delta, \rm eig}-u^*_{\rm eig}\Vert_{X}^2+\Vert \Pi^X_\delta u^*_{\rm eig}- u^*_{\rm eig}\Vert_{L^2(\Omega)}\big)\Vert u^*_{\delta, \rm eig} - \Pi^X_\delta u^*_{\rm eig} \Vert_{L^2(\Omega)}\\ \nonumber
 &+ \Vert u^*_{\delta, \rm eig} - \Pi^X_\delta u^*_{\rm eig}\Vert^2_{L^{2}(\Omega)}\\  \nonumber
 &\lsim \big(\delta \Vert u^*_{\rm eig} - \Pi_\delta^X u^*_{\rm eig}\Vert_{X}\Vert u^*_{\rm eig}\Vert_{H^2(\Omega)}+ \delta \Vert u^*_{\rm eig} - \Pi_\delta^X u^*_{\rm eig}\Vert_{X}\big)\Vert u^*_{\delta, \rm eig} - \Pi^X_\delta u^*_{\rm eig} \Vert_{L^2(\Omega)}\\ \nonumber
 &+ \Vert u^*_{\delta, \rm eig} - \Pi^X_\delta u^*_{\rm eig}\Vert^2_{L^{2}(\Omega)}\\ \label{eq:refine_est:polish_1}
& \lsim \delta \Vert u^*_{\rm eig} - \Pi_\delta^X u^*_{\rm eig}\Vert_{X}\Vert u^*_{\delta, \rm eig} - \Pi^X_\delta u^*_{\rm eig} \Vert_{L^2(\Omega)}+ \Vert u^*_{\delta, \rm eig} - \Pi^X_\delta u^*_{\rm eig}\Vert^2_{L^{2}(\Omega)}.
\end{align}

Applying now Estimate \eqref{new_proof_estimate_7}  to the above Estimate \eqref{eq:refine_est:polish_1} and making use of Young's inequality to replace some `mixed' terms, we deduce that for all $\delta >0$ sufficiently small it holds that
\begin{align}\nonumber
    \Vert u^*_{\delta, \rm eig} - \Pi^X_\delta u^*_{\rm eig} \Vert^2_{X}  &\lsim  \delta^2\Vert u^*_{\rm eig} - \Pi_\delta^X u^*_{\rm eig}\Vert_{X}\Vert u^*_{\delta, \rm eig} - \Pi^X_\delta u^*_{\rm eig}\Vert_X +  \delta^{q_{\rm eig}^*+2}\Vert u^*_{\rm eig} - \Pi^X_\delta u^*_{\rm eig}\Vert^2_X\\ \nonumber
    &+\delta^2\Vert u^*_{\delta, \rm eig} - \Pi^X_\delta u^*_{\rm eig}\Vert^2_X +  \delta^{2q_{\rm eig}^*+2}\Vert u^*_{\rm eig} - \Pi^X_\delta u^*_{\rm eig}\Vert_X^2\\ \label{eq:final-1}
    &\lsim  \delta^2\Vert u^*_{\delta, \rm eig} - \Pi^X_\delta u^*_{\rm eig}\Vert_X^2 +\delta^2 \Vert u^*_{\rm eig} - \Pi_\delta^X u^*_{\rm eig}\Vert_{X}\Vert u^*_{\delta, \rm eig} - \Pi^X_\delta u^*_{\rm eig}\Vert_X\\ \nonumber
    &+ \delta^{q_{\rm eig}^*+2}\Vert u^*_{\rm eig} - \Pi^X_\delta u^*_{\rm eig}\Vert^2_X.
\end{align}
Using now the fact that the term $\Vert u^*_{\delta, \rm eig} - \Pi^X_\delta u^*_{\rm eig}\Vert_X^2$ appears on the right-hand side of the above Estimate \eqref{eq:final-1} with pre-factor $\delta^2$ and appears on the left-hand side with no such $\delta$-dependent pre-factor, we deduce for all $\delta >0$ sufficiently large the simpler bound
\begin{align*}
    \Vert u^*_{\delta, \rm eig} - \Pi^X_\delta u^*_{\rm eig} \Vert^2_{X}  &\lsim  \delta^2 \Vert u^*_{\rm eig} - \Pi_\delta^X u^*_{\rm eig}\Vert_{X}\Vert u^*_{\delta, \rm eig} - \Pi^X_\delta u^*_{\rm eig}\Vert_X+ \delta^{q_{\rm eig}^*+2}\Vert u^*_{\rm eig} - \Pi^X_\delta u^*_{\rm eig}\Vert^2_X,
\end{align*}
i.e.,
\begin{align}\label{eq:quad_2'}
    \Vert u^*_{\delta, \rm eig} - \Pi^X_\delta u^*_{\rm eig} \Vert^2_{X}  &\le C_1 \delta^2 \Vert u^*_{\rm eig} - \Pi_\delta^X u^*_{\rm eig}\Vert_{X}\Vert u^*_{\delta, \rm eig} - \Pi^X_\delta u^*_{\rm eig}\Vert_X+ C_2\delta^{q_{\rm eig}^*+2}\Vert u^*_{\rm eig} - \Pi^X_\delta u^*_{\rm eig}\Vert^2_X,
\end{align}
for some positive constants $C_1$ and $C_2$, both independent of $\delta >0$.  

Notice now that, thanks to Young's inequality, for all $\delta >0$ we have the bound
\begin{align}\label{2K10}
C_1 \delta^2 \Vert u^*_{\rm eig} - \Pi_\delta^X u^*_{\rm eig}\Vert_{X}\Vert u^*_{\delta, \rm eig} - \Pi^X_\delta u^*_{\rm eig}\Vert_X \le \frac{C^2_1}{2}
\left[ \delta^4 \Vert u^*_{\rm eig} - \Pi_\delta^X u^*_{\rm eig}\Vert_{X}^2\right] + \frac{1}{2}
\Vert u^*_{\delta, \rm eig} - \Pi^X_\delta u^*_{\rm eig}\Vert_X^2.
\end{align}
Consequently, making use of Inequality \eqref{2K10} in Estimate \eqref{eq:quad_2'}, we deduce that for all $\delta>0$ small enough, it holds that
\begin{align}\nonumber
    \Vert u^*_{\delta, \rm eig} - \Pi^X_\delta u^*_{\rm eig} \Vert_{X} &\lsim \delta^2 \Vert u^*_{\rm eig} - \Pi_\delta^X u^*_{\rm eig}\Vert_{X} + \Big(\delta^{q_{\rm eig}^*+2}\Vert u^*_{\rm eig} - \Pi^X_\delta u^*_{\rm eig}\Vert^2_X\Big)^{1/2}\\[0.5em] \nonumber
    &\lsim \Big(\delta^2+\delta^{q_{\rm eig}^*/2+1}\Big) \Vert u^*_{\rm eig} - \Pi_\delta^X u^*_{\rm eig}\Vert_{X}.
\end{align}
Recalling again that the constant $q^*_{\rm eig} \in (0, 2]$ by \textbf{Assumption B.3} we finally arrive at the estimate
\begin{align}\label{eq:eig_lem_final}
    \Vert u^*_{\delta, \rm eig} - \Pi^X_\delta u^*_{\rm eig} \Vert_{X} \lsim \delta^{q_{\rm eig}^*/2+1}\Vert u^*_{\rm eig} - \Pi_\delta^X u^*_{\rm eig}\Vert_{X}.
\end{align}
We have thus obtained the sought-after estimate \eqref{eq:eig_higher_rates} in the $\Vert \cdot \Vert_X$ norm. To complete the proof, we require the corresponding estimate in the $\Vert \cdot \Vert_{L^2}$ norm. This can be obtained by making use of the above Estimate \eqref{eq:eig_lem_final} in the prior bound \eqref{new_proof_estimate_7} and utilising once again the fact that $q^*_{\rm eig} \in (0, 2]$ by \textbf{Assumption B.3}.
\end{proof}

\section{Application to Conforming Finite Element Discretisations}\label{sec:4}

Throughout this section, we assume that \textbf{Setting One} holds so that $X=H^1_0(\Omega)$. Let us now consider a family of regular, quasi-uniform triangulations $\{\mathcal{T}_h\}_{h > 0}$ of $\Omega$. Thus, for each $h>0$ and, for instance, $d=3$, the set  $\mathcal{T}_h$ consists of a collection of tetrahedra such that
\begin{itemize}
	\item $\overline{\Omega} = \cup_{K \in \mathcal{T}_h}K$;
	
	\item The intersection of any  $K, K' \in \mathcal{T}_{h}$ with $K\neq K'$ is either empty, a vertex, an entire edge or an entire face.
	
	\item The ratio $h_k$ of the diameter of any element $K \in \mathcal{T}_h$ to the diameter of its inscribed sphere is smaller than a constant independent of $h$. \vspace{1mm}

    \item The ratio $\frac{\max_{K \in \mathcal{T}_h} \vert K\vert }{\min_{K \in \mathcal{T}_h} \vert K\vert }$ is bounded by a constant independent of $h$.
\end{itemize}
Here, $h:= \max_{K \in \mathcal{T}_h} h_k$ and plays the role of the discretisation parameter.

Next, for each $K \in \mathcal{T}_h$ and each $n \in \mathbb{N}$, we denote by $\mathcal{P}_n(K)$ the space of restrictions to $K$ of algebraic polynomials of $d$ variables with maximum total degree $n$. Additionally, for each $h >0$, we define the finite element space $\mathcal{S}_{h, n}$ as the set of all continuous functions on $\overline{\Omega}$ such that the restriction of any element of $\mathcal{S}_{h, n}$ to an element $K \in \mathcal{T}_h$ belongs to $\mathcal{P}_n(K)$. For a given choice of polynomial degree $n \in \mathbb{N}$, we now define the sequence of approximation spaces $\{X_{h, n}\}_{h >0}$ that we consider in this section as
\begin{align}\label{eq:FEM_spaces}
\forall h >0\colon \qquad X_{h, n}:= \mathcal{S}_{h, n} \cap H_0^1(\Omega).
\end{align}

The application of Lemmas \ref{lem:combine_1} and \ref{lem:combine_2} to this choice of Galerkin approximation spaces results in the following theorem.

\begin{theorem}\label{thm:1a}
    Assume that \textbf{Setting One} holds, let $u^*_{\rm src}\in X$ denote the unique minimiser of the continuous energy functional \eqref{eq:1.1}, for a given polynomial degree $n \in \mathbb{N}$ let the finite element approximation spaces $\{X_{h, n}\}_{h >0}$ be defined according to Equation \eqref{eq:FEM_spaces}, let $u^*_{h,n, \rm src}\in X_{h, n}, ~ h>0$ denote the unique discrete minimiser of the discrete energy functional \eqref{eq:1.5}, and for every $h >0$ let $\Pi^X_{h, n} \colon X \rightarrow X_{h, n}$ denote the $(\cdot, \cdot)_X$-orthogonal projection operator onto $X_{h, n}$. Then for all $h>0$ sufficiently small and any $t \in [0, 3/2) \cap [0, \min\{r_v, n-1\}]$ we have the error estimate
    \begin{align}\label{eq:FEM_src_low}
    \Vert \Pi_{h, n}^X u^*_{\rm src}-u^*_{h,n, \rm src}\Vert_{H^1(\Omega)}&\lsim h\Vert \Pi_{h, n}^X u^*_{\rm src}-u^{*}_{\rm src}\Vert_{H^{1}(\Omega)} \hspace{4.2mm}\quad \text{if } n=1\\ \label{eq:FEM_src_high}
       \Vert \Pi_{h, n}^X u^*_{\rm src}-u^*_{h, n, \rm src}\Vert_{L^2(\Omega)} +h^{t/2}\Vert \Pi_{h, n}^X u^*_{\rm src}-u^*_{h, n, \rm src}\Vert_{H^1(\Omega)}&\lsim h^{t+1}\Vert \Pi_{h, n}^X u^*_{\rm src}-u^{*}_{ \rm src}\Vert_{H^{1}(\Omega)} \quad \text{if } n > 1.
    \end{align}
\end{theorem}
\begin{proof}
The proof is based on a verification of \textbf{Assumptions A.1, A.2,} and \textbf{A.3} which will allow us to apply Lemmas \ref{lem:combine_1} and \ref{lem:combine_2}. Note that these assumptions and lemmas have been formulated for approximation spaces indexed by the discretisation parameter $\delta >0$ so we identify $h=\delta$. 

The verification of \textbf{Assumption A.1} is trivial. Indeed, classical arguments yield that the Lagrangian~$\mathcal{P}_n$ finite elements satisfy the convergence estimate \eqref{eq:assum_1a} with rate (at least) $r_{\rm src}=2$, and it is also known (see, e.g., \cite[Proposition 1.134]{MR2050138}) that the orthogonal projection operator $\Pi_h^X\colon X \rightarrow X_h, ~ h>0$ satisfies the estimate \eqref{eq:assum_1b}.

The verification of \textbf{Assumption A.2} is also largely classical: since the exact minimiser $u^*_{\rm src} \in H^2(\Omega)$, it clearly suffices to show that
\begin{align*}
    \Vert u^*_{\rm src} - u^*_{h, n, \rm src} \Vert_{L^{\infty}(\Omega)} \qquad \text{is uniformly bounded in }h.
\end{align*}
To this end, we introduce (see, e.g., \cite[Section 1.4.2]{MR2050138}) the global interpolation operator $\mathcal{I}_{h, n}\colon \mathcal{C}^0(\overline{\Omega})\rightarrow X_{h, n}$ for Lagrangian $\mathcal{P}_n$ finite elements and deduce that for all $h >0$
\begin{align}\label{eq:new_FEM_1}
\Vert u^*_{\rm src} - u^*_{h, n, \rm src} \Vert_{L^{\infty}(\Omega)} &\leq \Vert u^*_{\rm src} - \mathcal{I}_{h, n} u^*_{\rm src} \Vert_{L^{\infty}(\Omega)}+\Vert \mathcal{I}_{h, n} u^*_{\rm src} - u^*_{h, n, \rm src} \Vert_{L^{\infty}(\Omega)}.
\end{align}
It is known (see, e.g, \cite[Theorem 4.4.20 and Remark 4.4.27]{MR2373954}) that the interpolation operator $\mathcal{I}_{h, n}$ satisfies the approximability property
\begin{align*}
\forall h> 0, ~\forall v \in H^{2}(\Omega) \colon \qquad \Vert v - \mathcal{I}_{h, n} v \Vert_{L^{\infty}(\Omega)} \lsim h^{1/2} \Vert v \Vert_{H^{2}(\Omega)}.
\end{align*}
Consequently, Inequality \eqref{eq:new_FEM_1} implies that for all $h > 0$ it holds that
\begin{align*}
\Vert u^*_{\rm src} - u^*_{h,n, \rm src} \Vert_{L^{\infty}(\Omega)} &\lsim h^{1/2} \Vert u^*_{\rm src} \Vert_{H^{2}(\Omega)} + \Vert \mathcal{I}_{h, n} u^*_{\rm src} - u^*_{h, n, \rm src} \Vert_{L^{\infty}(\Omega)}.
\end{align*}
To bound the second term on the right-hand side of the above inequality, we use the $L^{\infty}$-inverse inequality for Lagrangian $\mathcal{P}_n$ finite element approximation spaces (see, e.g., \cite[Lemma 1.142]{MR2050138}). This yields that for all $h$ sufficiently small, it holds that
\begin{align}\nonumber
\Vert u^*_{\rm src} - u^*_{h, n, \rm src} \Vert_{L^{\infty}(\Omega)} &\lsim h^{1/2} \Vert u^*_{\rm src} \Vert_{H^{2}(\Omega)} + h^{-1/2}\Vert \mathcal{I}_{h, n} u^*_{\rm src} - u^*_{h, n, \rm src} \Vert_{H^{1}(\Omega)}\\ \nonumber
&\lsim h^{1/2} \Vert u^*_{\rm src} \Vert_{H^{2}(\Omega)} + h^{-1/2}\Vert \mathcal{I}_{h, n} u^*_{\rm src} - u^*_{\rm src} \Vert_{H^{1}(\Omega)}+ h^{-1/2}\Vert u^*_{\rm src} - u^*_{h, n, \rm src} \Vert_{H^{1}(\Omega)}\\ \label{eq:new_FEM_2}
&\lsim h^{1/2} \Vert u^*_{\rm src} \Vert_{H^{2}(\Omega)},
\end{align}
where the last step follows from Estimate \eqref{eq:assum_1b} stated in \textbf{Assumption A.1} as well as the approximability properties of $\mathcal{I}_{h, n}$ in the $H^1(\Omega)$-norm (see, again, \cite[Theorem 4.4.20 and Remark 4.4.27]{MR2373954}). The verification of \textbf{Assumption A.2} now follows.

\vspace{1mm}

Having verified \textbf{Assumptions A.1} and \textbf{A.2}, we can now apply Estimate \eqref{eq:source_lower_rates} from Lemma \ref{lem:combine_2} to deduce the first sought-after Estimate \eqref{eq:FEM_src_low}. 

\medskip
To prove the second sought-after Estimate \eqref{eq:FEM_src_high}, we assume that the polynomial degree $n> 1$ and verify that \textbf{Assumption A.3} also holds. To this end, for each $h >0$ let the auxiliary potential function $\widetilde{V}_{h, n} \colon \Omega \rightarrow \mathbb{R}$ be given by (cf. Definition \ref{def:v_aux})
\begin{align*}
    \widetilde{V}_{h, n}:= V + \left(u^*_{\rm src}\right)^2 + \left(u^*_{h, n, \rm src}\right)^2 + u^*_{\rm src}u^*_{h, n, \rm src} -1. 
\end{align*}
We are interested in bounding, for any $h>0$ sufficiently small, all $\varphi\in H^2(\Omega)\cap X$ and all $\chi \in X$ the expression
    \begin{align*}
    \int_\Omega \widetilde{V}_{h, n}({\bf x})  \varphi({\bf x}) [\chi- \Pi^X_h \chi]({\bf x})\; d{\bf{x}}.
\end{align*}
To do so, we introduce, for each $h >0$, the following adjoint problem: we seek $ u \in X $ such that for all $v \in X$ it holds that
	\begin{equation}\label{def:adjoint_FEM_4}
	\int_\Omega  \nabla v({\bf{x}}) \cdot \nabla u ({\bf{x}})\; d {\bf{x}} = \int_\Omega \widetilde{V}_{h, n}({\bf{x}}) \varphi({\bf{x}}) v({\bf{x}}) d{\bf{x}}.
	\end{equation}	
\vspace{1mm}

We now make the following claims:\vspace{1mm}
\begin{description}
\item[Claim 1] For any $h > 0$ and any $t \in [0, 3/2) \cap [0, r_v]$, we have the inclusion \[\widetilde{V}_{h, n}\varphi \in H^{t}(\Omega) \cap H_0^1(\Omega)\cap \mathcal{C}^0(\overline{\Omega})\]. 
\vspace{-3mm}
\item[Claim 2] For any $t \in [0, 3/2) \cap [0, r_v]$, we have uniform boundedness of $\{\widetilde{V}_{h, n}\}_{h >0}$ in $H^{t}(\Omega)$.
\end{description}
\vspace{2mm}
In order to prove \textbf{Claim 1}, notice first that since $\varphi \in X = H_0^1(\Omega)$, it clearly suffices to show that
\begin{align*}
    \widetilde{V}_{h, n}\varphi \in H^{t}(\Omega)\cap \mathcal{C}^0(\overline{\Omega}) \qquad \forall t\in [0, 3/2) \cap [0, r_v].
\end{align*}
We first prove that $\widetilde{V}_{h, n}\varphi  \in \mathcal{C}^0(\overline{\Omega})$. To this end, let us recall the Sobolev embedding $H^{s}(\Omega) \hookrightarrow  \mathcal{C}^0(\overline{\Omega})$ for any $s >d/2$ (see, e.g., \cite[Corollary 9.15]{brezis2011functional}). Since, we additionally have $V \in H^{r_v}(\Omega)$ for some $r_v > d/2$, $u^*_{\rm src }\in H^2(\Omega)$ thanks to the elliptic regularity result \eqref{eq:elliptic_regularity}, and the trivial inclusion $\mathcal{S}_{h, n}\subset  \mathcal{C}^0{\overline{\Omega}}$, we conclude that indeed
\begin{align}\label{eq:c_inclusion}
    \widetilde{V}_{h, n}\varphi  \in \mathcal{C}^0(\overline{\Omega}).
\end{align}
Turning now to the inclusion in the Sobolev space, it suffices to notice that that the finite element approximation space $\mathcal{S}_{h, n} \in H^{t}(\Omega)$ for any $t < 3/2$, any $h >0$ and any $n \in \mathbb{N}$ (see, e.g., \cite[Theorem 2.9]{belgacem2001some}). Consequently, we can apply the inclusion \eqref{eq:c_inclusion} above together with the easily-verified estimate 
\begin{align}\label{eq:sobolev_product}
 \forall s\geq 0, \forall u, v \in H^s(\Omega) \cap L^\infty(\Omega) \colon  \qquad \Vert uv \Vert_{H^{s}(\Omega)}\lsim \bigl(\Vert u\Vert_{H^{s}(\Omega)}  + \Vert u\Vert_{L^{\infty}(\Omega)}\bigr) \bigl( \Vert v \Vert_{H^{s}(\Omega)} + \Vert v \Vert_{L^{\infty}(\Omega)} \bigr),
\end{align}
to conclude that indeed.
\begin{align*}
    \widetilde{V}_{h, n}\varphi  \in H^{t}(\Omega) \qquad \forall t\in [0, 3/2) \cap [0, r_v].
\end{align*}
This completes the proof of \textbf{Claim 1}.

\textbf{Claim 2} can be proven following arguments along the same lines as those of Remark \ref{rem:inverse}. To do so we require the availability of approximation rates for the global interpolation operator $\mathcal{I}_{h, n} \colon \mathcal{C}^0(\overline{\Omega}) \rightarrow X_{h, n}$ in the fractional Sobolev norm $\Vert \cdot \Vert_{H^t(\Omega)}$  for all $t \in [0, 3/2)$. These approximation rates can be found in \cite[Theorem 2.6]{belgacem2001some}. We also require inverse inequalities for the finite element approximation space $\mathcal{S}_{h, n}, ~h>0, n\in \mathbb{N}$, again with respect to the fractional Sobolev norm $\Vert \cdot \Vert_{H^t(\Omega)}$ for all $t \in [0, 3/2)$. The necessary inverse inequalities can be found in \cite[Theorem 2.9]{belgacem2001some}. For the sake of brevity, we desist from repeating the arguments.

Thanks to \textbf{Claim 1}, we can now appeal to Lemma \ref{lem:appendix} in the appendix, to deduce from elliptic regularity theory that, despite $\Omega$ being a corner domain, Equation \eqref{def:adjoint_FEM_4} has a unique solution $u \in X \cap H^{t+2}(\Omega)$ for any $t \in [0, 3/2) \cap [0, r_v]$, that satisfies the estimate
\begin{align}\label{eq:2.5_FEM*}
    \Vert u \Vert_{H^{2+t}(\Omega)} \lsim \Vert \widetilde{V}_{h, n} \varphi\Vert_{H^{t}(\Omega)}.
\end{align}
Making use of the adjoint problem \eqref{def:adjoint_FEM_4}, we deduce with the help of Estimate \eqref{eq:2.5_FEM*} that for all $t \in [0, 3/2) \cap [0, \min\{r_v, n-1\} ]$ it holds that
\begin{align}\nonumber
    \int_\Omega \widetilde{V}_{h, n}({\bf x})  \varphi({\bf{x}})[\chi-\Pi_{h, n}^{X}\chi]({\bf x})\; d{\bf{x}} &= \int_\Omega \nabla u ({\bf{x}}) \cdot \nabla [\chi-\Pi_N^{X}\chi]({\bf x})\; d {\bf{x}}\\ \nonumber
    &=\int_\Omega \nabla[u- \Pi_{h, n}^X u] ({\bf{x}}) \cdot  \nabla [\chi-\Pi_{h, n}^{X}\chi]({\bf x}) \; d {\bf{x}}\nonumber\\ \nonumber
&\leq \Vert u - \Pi_{h, n}^Xu\Vert_{X} \Vert \chi - \Pi_{h, n}^X\chi \Vert_{X}\\ \nonumber
&\lsim h^{1+t}\Vert u\Vert_{H^{2+t}(\Omega)} \Vert \chi - \Pi_{h, n}^X\chi \Vert_{X}\\ \nonumber
&\lsim h^{1+t}\Vert \widetilde{V}_{h, n}\vartheta\Vert_{H^{t}(\Omega)} \Vert \chi - \Pi_{h, n}^X\chi \Vert_{X}\\ \nonumber
\lsim h^{1+t}\Big(\bigl(\Vert \widetilde{V}_{h, n}\Vert_{H^{t}(\Omega)} &+ \Vert \widetilde{V}_{h, n}\Vert_{L^{\infty}(\Omega)}\bigr)\bigl(\Vert \vartheta\Vert_{H^{t}(\Omega)}+\Vert \vartheta\Vert_{L^{\infty}(\Omega)}\bigr)\Big) \Vert \chi - \Pi_{h, n}^X\chi \Vert_{X}.
\end{align}
Here, we have used the higher-order convergence rates (that depend on~$n$) satisfied by the $(\cdot, \cdot)_X$-orthogonal projection operator $\Pi_{h, n}^X$, which can, e.g., be found in \cite[Theorem 3.16]{ern2004theory}, and the last step follows from Estimate \eqref{eq:sobolev_product}. Recalling now \textbf{Claim 2}, Lemma \ref{lem:combine_1} and the inclusion $\vartheta \in H^2(\Omega)$, we deduce that for all $h > 0$ sufficiently small, all $n \in \mathbb{N}$, and all $t \in [0, 3/2) \cap [0, \min\{r_v, n-1\} ]$ it holds that
\begin{align}
 \int_\Omega \widetilde{V}_{h, n}({\bf x})  \varphi({\bf{x}})[\chi-\Pi_N^{X}\chi]({\bf x})\; d{\bf{x}} \lsim  h^{1+t} \Vert \chi - \Pi_{h, n}^X\chi \Vert_{X}.
\end{align}
It therefore follows that \textbf{Assumption A.3} also holds with any constant $q^*_{\rm src} \in [0, 3/2) \cap [0, \min\{r_v, n-1\} ]$. We can thus apply Lemma \ref{lem:combine_2} to conclude.
\end{proof}

In a similar fashion, the application of Lemma \ref{lem:combine_eig_2} to the present choice of conforming finite element Galerkin approximation spaces results in the following theorem.
\begin{theorem}\label{thm:1b}
    Assume that \textbf{Setting One} holds, let $u^*_{\rm eig}\in X$ denote the unique positive minimiser of the continuous energy functional \eqref{eq:1.2}, for a given polynomial degree $n \in \mathbb{N}$ let the finite element approximation spaces $\{X_{h, n}\}_{h >0}$ be defined according to Equation \eqref{eq:FEM_spaces}, let $u^*_{h, n, \rm eig}\in X_{h, n}, ~ h>0$ denote a sequence of discrete minimisers of the discrete energy functional \eqref{eq:1.5} that satisfy $\big(u^*_{h, n, \rm eig}, u^*_{\rm eig}\big)_{L^2(\Omega)}\geq 0$, and for every $h >0$ let $\Pi^X_{h, n} \colon X \rightarrow X_h$ denote the $(\cdot, \cdot)_X$-orthogonal projection operator onto $X_{h, n}$. Then for all $h>0$ sufficiently small and any $t \in [0, 3/2) \cap [0, \min\{r_v, n-1\}]$ we have the error estimates
    \begin{align}\label{eq:FEM_eig_low}
    \Vert \Pi_{h, n}^X u^*_{\rm eig}-u^*_{h,n, \rm eig}\Vert_{H^1(\Omega)}&\lsim h\Vert \Pi_{h, n}^X  u^*_{\rm eig}-u^{*}_{\rm eig}\Vert_{H^{1}(\Omega)} \hspace{3.99mm}\quad \text{if } n=1\\ \label{eq:FEM_eig_high}
       \Vert \Pi_{h, n}^X u^*_{\rm eig}-u^*_{h, n, \rm eig}\Vert_{L^2(\Omega)} +h^{t/2}\Vert \Pi_{h, n}^X u^*_{\rm eig}-u^*_{h, n, \rm eig}\Vert_{H^1(\Omega)}&\lsim h^{t+1}\Vert \Pi_{h, n}^X u^*_{\rm eig}-u^{*}_{\rm eig}\Vert_{H^{1}(\Omega)} \quad \text{if } n >1.
    \end{align}
\end{theorem}

\begin{proof}
    Analogously to the proof of Theorem \ref{thm:1a} above, the proof of the present theorem is based on a verification of \textbf{Assumptions B.1, B.2,} and \textbf{B.3} which allow us to apply Lemma \ref{lem:combine_eig_2}. With the exception of Estimate \eqref{eq:assum_4new} in \textbf{Assumption B.3}, these verifications are virtually identical to the ones carried out for \textbf{Assumptions A.1, A.2,} and \textbf{A.3} in the proof of Theorem \ref{thm:1a} so we refrain from repeating the arguments but we refer the interested reader to \cite{cances2010numerical} for the tools necessary for these verifications.

    The verification of Estimate \eqref{eq:assum_4new} in \textbf{Assumption B.3} is similar in spirit to that of Estimate \eqref{eq:assum_4b} in \textbf{Assumption A.3} as we now briefly discuss. To begin with, let us introduce the following adjoint problem: we seek $\vartheta \in X$ that solves the equation 
 \begin{align}\label{eq:adjoint_eigen_FEM}
\forall v \in X\colon \quad \left\langle \big(\mathcal{E}_{\rm eig}''[u^*_{\rm eig}] -\lambda^* \big)v, \vartheta\right\rangle_{X^* \times X}= \big(u^*_{\rm eig}, v\big)_{L^2(\Omega)}.     
 \end{align}
 Since $\mathcal{E}_{\rm eig}''[u^*_{\rm eig}]  \colon X \rightarrow X^*$ is self-adjoint, and we have the coercivity and continuity properties~\eqref{eq:shifted_energy_estimates}, we immediately deduce that the adjoint problem \eqref{eq:adjoint_eigen_FEM}  has a unique solution $\vartheta \in X$.  Moreover, since we have the Sobolev embedding $H^2 (\Omega)\hookrightarrow \mathcal{C}^0(\overline{\Omega})$, we have the inclusion $u^*_{\rm eig}\in H^2(\Omega) \cap X \cap \mathcal{C}^0(\overline{\Omega})$. We can therefore appeal to Lemma \ref{lem:appendix} in the appendix, to deduce from elliptic regularity theory and a bootstrapping argument that, despite $\Omega$ being a corner domain, the solution $\vartheta$ to Equation \eqref{eq:adjoint_eigen} is an element of $H^{2+t}(\Omega)\cap X$ and satisfies the estimate  
\begin{align}\label{eq:FEM_eig*}
 \forall t \in [0, \min\{2, r_v\}] \colon \qquad  \Vert \vartheta \Vert_{H^{2+t}(\Omega)} \lsim \Vert u^*_{\rm eig}\Vert_{H^{t}(\Omega)}
\end{align}
Note that since $u^*_{\rm eig}$ satisfies the elliptic equation 
\begin{align*}
-\Delta u_{\rm eig}^* = \lambda^*u^*_{\rm eig} - V u^*_{\rm eig} -(u^*_{\rm eig})^3    \qquad \text{in } X,
\end{align*}
the same elliptic regularity argument yields that, in fact, $u^*_{\rm eig} \in H^{2+t}(\Omega)$ for every $t \in [0, \min\{2, r_v\}]$ and satisfies a similar estimate as \eqref{eq:FEM_eig*}. This fact will be of use in the sequel.
 
 Returning now to the adjoint equation \eqref{eq:adjoint_eigen_FEM}, we see that the choice of test function $v=u^*_{h, n, \rm eig}-u^*_{\rm eig}$, implies that
     \begin{align}\label{eq:final_FEM_eig_-1}
\left(u^*_{\rm eig}, u^*_{h, n,\rm eig}-u^*_{\rm eig}\right)_{L^2(\Omega)} &=\left\langle \big(\mathcal{E}_{\rm eig}''[u^*_{\rm eig}] -\lambda^* \big)u^*_{h, n,\rm eig}-u^*_{\rm eig}, \vartheta\right\rangle_{X^* \times X}\\ \nonumber
&=\underbrace{\left\langle \big(\mathcal{E}_{\rm eig}''[u^*_{\rm eig}] -\lambda^* \big)u^*_{h, n,\rm eig}-u^*_{\rm eig}, \vartheta-\Pi_{h, n}^X\vartheta \right\rangle_{X^* \times X}}_{:=\rm (I)}\\ \nonumber
&\underbrace{\left\langle \big(\mathcal{E}_{\rm eig}''[u^*_{\rm eig}] -\lambda^* \big)u^*_{h, n,\rm eig}-u^*_{\rm eig}, \Pi_{h, n}^X\vartheta \right\rangle_{X^* \times X}}_{:= \rm (II)}.
 \end{align}

We begin with the term (I). Using the continuity of the shifted second derivative $\mathcal{E}_{\rm eig}''[u^*_{\rm eig}]  -\lambda^*$ together with the higher-order convergence rates (that depend on $n$) satisfied by the $(\cdot, \cdot)_X$-orthogonal projection operator $\Pi_{h, n}^X$ (see, e.g., \cite[Theorem 3.16]{ern2004theory}), we deduce that for any $t \in [0, \min\{2, r_v, n-1\}]$ and all $h>0$ sufficiently small it holds that
\begin{align}\nonumber
    {\rm (I)} \lsim \Vert u^*_{h, n,\rm eig}-u^*_{\rm eig}\Vert_X \Vert \vartheta-\Pi_{h, n}^X\vartheta \Vert_X &\lsim h^{1+t}\Vert u^*_{h, n,\rm eig}-u^*_{\rm eig}\Vert_X \Vert \vartheta \Vert_{H^{2+t}(\Omega)}\\ \label{eq:final_FEM_eig_0}
    &\lsim h^{1+t}\Vert u^*_{h, n,\rm eig}-u^*_{\rm eig}\Vert_X \Vert u^*_{\rm eig}\Vert_{H^{t}(\Omega)}.
\end{align}
Here, the last step follows from Estimate \eqref{eq:FEM_eig*}.

Turning now to the term (II), as argued in the proof of Lemma \ref{lem:combine_eig_2} (see Equation \eqref{eq:discret_eigen}), we have the relation
\begin{equation}\label{eq:discret_eigen_FEM}
\begin{aligned}
{\rm (II)}=\Big \langle \big(\mathcal{E}_{\rm eig}''[u^*_{\rm eig}]-\lambda^*\big)(u^*_{h,n, \rm eig} &- u^*_{\rm eig}), \Pi_{h, n}^X\vartheta\Big\rangle_{X^* \times X}=(\lambda^*_{h, n}-\lambda^*)\int_\Omega u^*_{h,n, \rm eig}(\bold{x}) \Pi_{h, n}^X\vartheta (\bold{x}) \; d\bold{x}\\[0.25em]
 &-\int_\Omega \big(u^*_{h, n, \rm eig}(\bold{x})-u^*_{\rm eig}(\bold{x})\big)^2\big(u^*_{h,n, \rm eig}(\bold{x})+2u^*_{\rm eig}(\bold{x})\big)\Pi_{h, n}^X\vartheta(\bold{x})\; d \bold{x}.
\end{aligned}
\end{equation}

Before proceeding with our analysis of Estimate \eqref{eq:discret_eigen_FEM}, we first study the stability of the $(\cdot, \cdot)_X$ projection operator in the $L^{\infty}(\Omega)$ norm. This can be done using the same arguments used to verify \textbf{Assumption A.2} in the proof of Theorem \ref{thm:1a}. Indeed, recalling again the global interpolation operator $\mathcal{I}_{h, n} \colon \mathcal{C}^0(\overline{\Omega})\rightarrow X_{h, n}$ introduced in the proof of Theorem \ref{thm:1a} and using the fact that $\vartheta \in H^2(\Omega)\cap X$ by hypothesis, we deduce that for all $h >0$ it holds that
\begin{align}\nonumber
    \Vert\Pi_{h, n}^X\varphi\Vert_{L^\infty(\Omega)} &\leq  \Vert  \mathcal{I}_{h, n}\varphi\Vert_{L^\infty(\Omega)} + \Vert \mathcal{I}_{h, n}\varphi- \Pi_{h, n}^X\varphi\Vert_{L^\infty(\Omega)}\\ \nonumber
    &\lsim \Vert \varphi\Vert_{L^\infty(\Omega)} + h^{-1/2}\Vert \mathcal{I}_{h, n}\varphi- \Pi_{h, n}^X \varphi\Vert_{H^1(\Omega)}\\ \nonumber
    &\lsim \Vert \varphi\Vert_{H^2(\Omega)} + h^{-1/2}\Vert \mathcal{I}_{h, n}\varphi- \varphi\Vert_{H^1(\Omega)}+ h^{-1/2}\Vert \varphi- \Pi_{h, n}^X \varphi\Vert_{H^1(\Omega)}\\
    &\lsim \Vert \varphi\Vert_{H^2(\Omega)}
    .
\end{align}

Consequently, simple applications of Hölder's inequality together with Estimate \eqref{eq:FEM_eig*} allow us to deduce from the above Equation \eqref{eq:discret_eigen_FEM} that the term (II) simplifies as
\begin{align}\nonumber
   {\rm (II)}  &\lsim \vert\lambda^*_{h, n}-\lambda^* \vert \Vert  u^*_{h,n, \rm eig}\Vert_{L^2(\Omega)}\Vert \Pi_{h, n}^X\vartheta\Vert_{L^2(\Omega)} + \Vert u^*_{h, n, \rm eig}-u^*_{\rm eig}\Vert^2_{L^2(\Omega)}\Vert \big(u^*_{h,n, \rm eig}+2u^*_{\rm eig}\big)\Pi_{h, n}^X\vartheta\Vert_{L^\infty(\Omega)}\\ \nonumber
   &\lsim \vert\lambda^*_{h, n}-\lambda^* \vert \Vert\vartheta\Vert_{H^2(\Omega)} + \Vert u^*_{h, n, \rm eig}-u^*_{\rm eig}\Vert^2_{L^2(\Omega)}\Vert u^*_{h,n, \rm eig}+2u^*_{\rm eig}\Vert_{L^{\infty}(\Omega)}\Vert \vartheta\Vert_{H^2(\Omega)}\\ \nonumber
   &\lsim \vert\lambda^*_{h, n}-\lambda^* \vert \Vert u^*_{\rm eig}\Vert_{H^2(\Omega)} + \Vert u^*_{h, n, \rm eig}-u^*_{\rm eig}\Vert^2_{L^2(\Omega)}\Vert u^*_{h,n, \rm eig}+2u^*_{\rm eig}\Vert_{L^{\infty}(\Omega)}\Vert u^*_{\rm eig}\Vert_{H^2(\Omega)}\\ \label{eq:final_FEM_eig_1}
   &\lsim \vert\lambda^*_{h, n}-\lambda^* \vert \Vert u^*_{\rm eig}\Vert_{H^2(\Omega)} + \Vert u^*_{h, n, \rm eig}-u^*_{\rm eig}\Vert^2_{L^2(\Omega)}\Vert u^*_{\rm eig}\Vert_{H^2(\Omega)},
\end{align}
where the last step uses the uniform boundedness of $\{u_{h, n, \rm eig}^*\}_{h >0}$ yielded by \textbf{Assumption B.2}. Using the finite element convergence rates stated in \cite{cances2010numerical} (see in particular \cite[Inequality 65, Theorem 3]{cances2010numerical}) and recalling that -- due to elliptic regularity -- $u^*_{\rm eig} \in H^{2+t}(\Omega)$ for every $t \in [0, \min\{2, r_v\}]$, we can further deduce from Estimate \eqref{eq:final_FEM_eig_1} that for all  $t \in [0, 3/2) \cap [0, \min\{2, r_v, n-1\}]$ and all $h>0$ sufficiently small it holds that
\begin{align}\nonumber
    {\rm (II)}  &\lsim  h^{1+t} \Vert u^*_{h, n,\rm eig}-u^*_{\rm eig}\Vert_X\Vert u^*_{\rm eig}\Vert_{H^2(\Omega)} + h^{3+t}\Vert u^*_{h, n, \rm eig}-u^*_{\rm eig}\Vert_{X}\Vert u^*_{\rm eig}\Vert_{H^2(\Omega)}\\
    &\lsim  h^{1+t} \Vert u^*_{h, n,\rm eig}-u^*_{\rm eig}\Vert_X\Vert u^*_{\rm eig}\Vert_{H^2(\Omega)}. \label{eq:final_FEM_eig_2}
\end{align}
Here, the additional restriction that $t\in [0, 3/2)$ arises for technical reasons related to the convergence rate of $\vert \lambda^*_{h, n}-\lambda^*\vert$. We refer to the proof of \cite[Inequality 65, Theorem 3]{cances2010numerical} for details\footnote{Essentially, one uses a duality argument and the fact that $u^*_{\rm eig}\in H^{t}(\Omega)\cap H_0^1(\Omega)$ to bound the error $\vert \lambda^*_{h, n}-\lambda^*\vert$ by the error $u^*_{h, n, \rm eig}-u^*_{\rm eig}$ in the negative-order Sobolev space $H^{-t}(\Omega)$. But $(H^{t}(\Omega)\cap H_0^1(\Omega))'= H^{-t}(\Omega)$ only if $t \in [0, 3/2)$.}.

Combining Estimates \eqref{eq:final_FEM_eig_0} and \eqref{eq:final_FEM_eig_2} with the earlier Equation \eqref{eq:final_FEM_eig_-1} we finally conclude that for all  $t \in [0, 3/2) \cap [0, \min\{r_v, n-1\}]$ and all $h>0$ sufficiently small it holds that
\begin{align*}
\left(u^*_{\rm eig}, u^*_{h, n,\rm eig}-u^*_{\rm eig}\right)_{L^2(\Omega)}\lsim h^{1+t} \Vert u^*_{h, n,\rm eig}-u^*_{\rm eig}\Vert_X\Vert u^*_{\rm eig}\Vert_{H^2(\Omega)}.
 \end{align*}
It therefore follows that Estimate \eqref{eq:assum_4new} in \textbf{Assumption B.3} also holds with any constant $q^*_{\rm src} \in [0, 3/2) \cap [0, \min\{r_v, n-1\}]$.

\end{proof}

\section{Application to Spectral Polynomial Discretisations}\label{sec:5}

As in Section \ref{sec:4}, we assume \textbf{Setting One} throughout the present section so that $X=H^1_0(\Omega)$. Let us now introduce for each $N \in \mathbb{N}$, the set $\mathbb{P}_N(\Omega)$ as the space of restrictions to $\Omega$ of algebraic polynomials of degree $\leq N$ in each of its $d$ variables. We then define the sequence of approximation spaces $\{X_{N}\}_{N \in \mathbb{N}}$ that we consider in this section~as
\begin{align}\label{eq:polynomial_spaces}
\forall N \in \mathbb{N} \colon \quad X_{N}:= \mathbb{P}_N(\Omega) \cap H_0^1(\Omega).
\end{align}
Thus, $N^{-1}$ plays the role of the discretisation parameter $\delta$ from Sections \ref{sec:source} and \ref{sec:eig}.

The main result that we state in this section will require slightly stronger regularity assumptions on the effective potential $V$ and source function $f$, namely, that $V \in H^{r_v}(\Omega)$ for some $r_v \geq2$ and $f \in H^s(\Omega)$ for some $s \geq 1$. While our results can be reformulated to apply also in the lower regularity regime, this introduces needless expositional overhead. Additionally, since spectral methods are typically applied in problems where the Sobolev regularity of the input data is high, we feel that this additional assumption is not unreasonable.

Equipped with this additional regularity assumption, we can now apply Lemmas \ref{lem:combine_1} and \ref{lem:combine_2} to the Galerkin approximation spaces above to obtain the following theorem.

\begin{theorem}\label{thm:2a}
    Assume that \textbf{Setting One} holds and $V \in H^{r_v}(\Omega)$ and $f \in H^s(\Omega)$ for some $r_v \geq2$ and $s \geq 1$ respectively, let $u^*_{\rm src}\in X$ denote the unique minimiser of the continuous energy functional~\eqref{eq:1.1}, let the spectral polynomial approximation spaces $\{X_{N}\}_{N \in \mathbb{N}}$ be defined according to Equation~\eqref{eq:polynomial_spaces}, let $u^*_{N, \rm src}\in X_N, ~ N \in \mathbb{N}$ denote the unique discrete minimiser of the discrete energy functional~\eqref{eq:1.5}, and for every $N \in \mathbb{N}$ let $\Pi^X_N \colon X \rightarrow X_N$ denote the $(\cdot, \cdot)_X$-orthogonal projection operator onto $X_N$. Then $\forall \epsilon> 0$ small enough we have that 
    \begin{align}\label{eq:Poly_src}
      \Vert \Pi_N^X u^*_{\rm src}-u^{*}_{N, \rm src}\Vert_{L^2(\Omega)} +N^{-1+{\epsilon}/2}\Vert \Pi_N^X u^*_{\rm src}-u^{*}_{N, \rm src}\Vert_{H^1(\Omega)}&\lsim N^{-3 + \epsilon}\Vert \Pi_N^X u^*_{\rm src}- u^{*}_{\rm src}\Vert_{H^{1}(\Omega)}.
    \end{align}
\end{theorem}
\begin{proof}
Similar to the proof of Theorem \ref{thm:1a}, the proof of the present theorem is  based on a verification of \textbf{Assumptions A.1, A.2,} and \textbf{A.3} followed by an application of Lemmas \ref{lem:combine_1} and \ref{lem:combine_2}.

    Turning first to \textbf{Assumption A.1}, we recall that it follows from standard elliptic regularity results on polygonal and polyhedral domains (see, e.g., \cite[Remark 2.1]{MR1470226} for the case $d=2$ and \cite[Chapter 6, Section 18.F]{MR0961439} for the case $d=3$), that $u^*_{\rm src} \in H^{3-\epsilon}(\Omega)$ for every $\epsilon >0$. Consequently, we may apply well-known convergence estimates for spectral polynomial methods to deduce that the sought-after estimate \eqref{eq:assum_1a} holds for any $r_{\rm src}= 3-\epsilon$ with $\epsilon >0$ sufficiently small\footnote{In the case $d=1$, we can take $r_{\rm src}=3$ since $\Omega= (-1, 1)$ has no corners and therefore $u^*_{\rm src} \in H^{3}(\Omega)$. Note also that, if the source function $f$ vanishes on the boundary, then the first-order singularity in the solution is skipped and the solution $u^{*}_{\rm src}$ is more regular (see, e.g., Lemma \ref{lem:appendix} in the appendix.)}. It is also known (see, e.g., \cite[Theorems 7.2 and 8.3]{MR1470226}) that the orthogonal projection operator $\Pi_N^X\colon X \rightarrow X_N, ~ N \in \mathbb{N}$ satisfies the estimate \eqref{eq:assum_1b}.

Next, we focus on the verification of \textbf{Assumption A.2}. To do so, we will follow the strategy outlined in Remark \ref{rem:inverse} and establish a stronger result, namely that
\begin{align}\label{eq:revision_hassan_new_verify_spectral}
\exists p^* \in (d/2, 2), ~ \exists C>0, ~\forall N\in \mathbb{N} \colon \qquad \Vert u^*_{N, \rm src} \Vert_{H^{p^*}(\Omega)} \leq C.
\end{align}
As argued in Remark \ref{rem:inverse}, thanks to the Sobolev embedding theorem, Estimate~\eqref{eq:revision_hassan_new_verify_spectral} will imply the uniform $L^\infty(\Omega)$-boundedess of the discrete minimisers $\{u^*_{N, \rm src}\}_{N \in \mathbb{N}}$ as required by \textbf{Assumption A.2}. Additionally, Estimate \eqref{eq:revision_hassan_new_verify_spectral} will also be useful in the verification of \textbf{Assumption A.3}.

In view of Remark \ref{rem:inverse}, to prove Estimate \eqref{eq:revision_hassan_new_verify_spectral}, it suffices to establish that

 \begin{enumerate} 
 
 \item There exists $ s^*_{\rm src}\geq2$ such that for every $N \in \mathbb{N}$ sufficiently small there exists an operator $\Pi_N^{2} \colon H^2(\Omega)\cap X \rightarrow X_N$ such that for all real numbers $0\leq t \leq 2$ and any $v \in H^{s^*_{\rm src}}(\Omega) \cap X$ it holds that
        \begin{align}\label{eq:assum_3primeprime_new}
            \Vert v - \Pi_N^{2} v \Vert_{H^{t}(\Omega)} \lsim N^{t-s^*_{\rm src}} \Vert v \Vert_{H^{s^*_{\rm src}}(\Omega)};
        \end{align}

        \item There exists a mapping $\theta \colon [1, +\infty) \rightarrow \mathbb{R}$ such that 
    \begin{equation}\label{eq:assum_3prime_new}
    \begin{split}
        \forall p \geq 1, ~ \forall v_N \in X_N \cap H^{p}(\Omega)\colon& \qquad \Vert v_N \Vert_{H^{p}(\Omega)} \lsim N^{\theta(p)} \Vert  v_N \Vert_{H^1(\Omega)} \qquad \text{and}\\
\exists p^* \in (d/2, 2)\colon& \qquad \theta(p^*)+1- \min\{s^*_{\rm src}, r_{\rm src}\}< 0.
    \end{split}
    \end{equation}

\end{enumerate}

To verify the first condition above, we note that according to \cite[Theorem 7.4]{MR1470226}, there exists an operator $\widetilde{\Pi}^{2, 1, 0}_N\colon H^2(\Omega)\cap X \rightarrow X_N$ with the property that for all $0\leq t \leq 2 \leq s$ and all $v \in  H^s(\Omega)\cap X$ it holds that
\begin{align}\label{eq:2.1*}
    \Vert v - \widetilde{\Pi}^{2, 1, 0}_N v\Vert_{H^{t}(\Omega)} \lsim N^{t-s} \Vert v\Vert_{H^{s}(\Omega)}.
\end{align}
Consequently, Estimate \eqref{eq:assum_3primeprime_new} holds with $s^*_{\rm src}= r^*_{\rm src} > 2$ and $ \Pi_N^{2}= \widetilde{\Pi}^{2, 1, 0}_N$. 

Turning now to Estimate \eqref{eq:assum_3prime_new}, we recall that the spectral polynomial approximation spaces $\{X_N\}_{N \in \mathbb{N}}$ are known to satisfy the inverse inequality
\begin{align}\label{eq:2.3*}
 \forall p\in [1, \infty), \forall v_N \in X_N \colon \qquad   \Vert v_N\Vert_{H^{p}(\Omega)} \lsim N^{2(p-1)}\Vert v_N\Vert_{H^1(\Omega)}.
\end{align}
A proof of this inverse inequality for the case $d=1$, i.e., when $\Omega=(-1, 1)$ can be found in \cite[Theorem 5.2]{MR1470226}. Since we assume that $\Omega$ is a tensorised domain of the form $\Omega=(-1, 1)^d$, the proof of Inequality \eqref{eq:2.3*} for the cases $d=2$ and $d=3$ follow as a straightforward consequence of the one-dimensional result. Taking therefore $\theta (p)= 2(p-1)$, recalling that $s^*_{\rm src}= r^*_{\rm src}$ and $r^*_{\rm src}$ can be taken as $r^*_{\rm src}=3-\epsilon$ for any $\epsilon>0$ sufficiently small, we see that Estimate \eqref{eq:assum_3prime_new} holds for any $p^* \in (d/2, 2)$. Thus, the uniform bound \eqref{eq:revision_hassan_new_verify_spectral} and therefore \textbf{Assumption A.2} holds.

Let us now turn to the verification of \textbf{Assumption A.3}. For each $N \in \mathbb{N}$ let the auxiliary potential function $\widetilde{V}_N \colon \Omega \rightarrow \mathbb{R}$ be given by (cf. Definition \ref{def:v_aux})
\begin{align*}
    \widetilde{V}_N:= V + \left(u^*_{\rm src}\right)^2 + \left(u^N_{\rm src}\right)^2 + u^*_{\rm src}u^N_{\rm src} -1. 
\end{align*}
We are now interested in bounding, for any $N \in \mathbb{N}$ sufficiently large, all $\varphi\in H^2(\Omega)\cap X$ and all $\chi \in X$ the expression
    \begin{align*}
    \int_\Omega \widetilde{V}_N({\bf x})  \varphi({\bf x}) [\chi- \Pi^X_N \chi]({\bf x})\; d{\bf{x}}.
\end{align*}
To do so, we will follow essentially the same approach as that taken in the proof of Theorem \ref{thm:1a}. To start with, we claim that
\begin{align}\label{eq:spectral_extension}
\forall N \in \mathbb{N}\colon \qquad \widetilde{V}_N \phi \in H^2(\Omega)\cap X \cap \mathcal{C}^0(\overline{\Omega})
\end{align}
Indeed, this follows directly from the Sobolev embedding $H^2(\Omega) \hookrightarrow \mathcal{C}^0(\overline{\Omega})$ together with the fact that the effective potential $V \in H^{r_v}(\Omega)$ for some $r_v \geq 2$.

Next, we introduce, for each $N \in \mathbb{N}$, the following adjoint problem: we seek $ u \in X $ such that for all $v \in X$ it holds that
	\begin{equation}\label{def:adjoint_4}
	\int_\Omega  \nabla v({\bf{x}}) \cdot \nabla u ({\bf{x}})\; d {\bf{x}} = \int_\Omega \widetilde{V}_N({\bf{x}}) \varphi({\bf{x}}) v({\bf{x}}) d{\bf{x}}.
	\end{equation}	

Clearly Equation~\eqref{def:adjoint_4} has a unique solution $u \in X$, and thanks to Estimate \eqref{eq:spectral_extension}, we can appeal to Lemma \ref{lem:appendix} in the appendix to deduce from elliptic regularity theory that this solution $u\in H^4(\Omega)$ and satisfies for any $s \in [0, 2] $, the estimate
\begin{align}\label{eq:2.5*}
    \Vert u \Vert_{H^{2+s}(\Omega)} \lsim \Vert \widetilde{V}_{N} \varphi\Vert_{H^{s}(\Omega)}.
\end{align}

Plugging in the test function $v = \chi - \Pi_N^X \chi$ in the adjoint problem \eqref{def:adjoint_4}, we now deduce that for all $N \in \mathbb{N}$ and all $s \in (d/2, 2)$, it holds that
\begin{align}\nonumber
    \int_\Omega \widetilde{V}_N({\bf x})  \varphi({\bf{x}})[\chi-\Pi_N^{X}\chi]({\bf x})\; d{\bf{x}} &= \int_\Omega \nabla u ({\bf{x}}) \cdot \nabla [\chi-\Pi_N^{X}\chi]({\bf x})\; d {\bf{x}}\\ \nonumber
    &=\int_\Omega \nabla[u- \Pi_N^X u] ({\bf{x}}) \cdot  \nabla [\chi-\Pi_N^{X}\chi]({\bf x}) \; d {\bf{x}}\\ \nonumber
&\leq \Vert u- \Pi^X_N u\Vert_{H^1(\Omega)} \Vert  \chi-\Pi_N^{X}\chi\Vert_{H^1(\Omega)}\\ \nonumber
&\lsim N^{-s-1}\Vert u\Vert_{H^{2+s}(\Omega)} \Vert  \chi-\Pi_N^{X}\chi\Vert_{H^1(\Omega)}\\ \nonumber
&\lsim N^{-s-1}\Vert \widetilde{V}_N \varphi\Vert_{H^{s}(\Omega)} \Vert  \chi-\Pi_N^{X}\chi\Vert_{H^1(\Omega)}\\ \nonumber
& \lsim N^{-s-1}\Vert \widetilde{V}_N\Vert_{H^{s}(\Omega)} \Vert \varphi\Vert_{H^{2}(\Omega)} \Vert  \chi-\Pi_N^{X}\chi\Vert_{H^1(\Omega)}.
\end{align}
Here, we have used the higher-order convergece rates enjoyed by the $(\cdot, \cdot)_X$-orthogonal projection operator $\Pi_N^X$ in the spectral polynomial approximation spaces $\{X_N\}_{N \in \mathbb{N}}$  (see, e.g., \cite[Theorems 7.2 and 8.3]{MR1470226}), and the last step follows from the fact that $H^{s}(\Omega), ~ s\in (d/2, 2)$ is a Banach algebra and we have the continuous embedding $H^2(\Omega) \hookrightarrow H^{s}(\Omega)$.

Finally notice that by assumption, the effective potential $V \in H^{r_v}(\Omega)$ for some $r_v \geq 2$ so that  Estimate \eqref{eq:elliptic_regularity} and \eqref{eq:revision_hassan_new_verify_spectral} together imply that $\{\widetilde{V}_N\}_{N \in \mathbb{N}}$ is uniformly bounded in $H^{s}(\Omega)$ for every $s \in (d/2, 2)$. It therefore follows that \textbf{Assumption A.3} also holds with any constant $q^*_{\rm src} \in (d/2, 2)$. In particular, we can apply Lemma \ref{lem:combine_2} with $q^*_{\rm src}=2-\epsilon$ for any $\epsilon >0$ to conclude.
\end{proof}

In a similar fashion, the application of Lemma \ref{lem:combine_eig_2} to the present choice of spectral polynomial Galerkin approximation spaces results in the following theorem.
\begin{theorem}\label{thm:2b}
    Assume that \textbf{Setting One} holds and $V \in H^{r_v}(\Omega)$ for some $r_v \geq2$, let $u^*_{\rm eig}\in X$ denote the unique positive minimiser of the continuous energy functional \eqref{eq:1.1}, let the spectral polynomial approximation spaces $\{X_{N}\}_{N \in \mathbb{N}}$ be defined according to Equation \eqref{eq:polynomial_spaces}, let $u^*_{N, \rm eig}\in X_N, ~ N \in \mathbb{N}$ denote a sequence of discrete minimisers of the discrete energy functional \eqref{eq:1.5} that satisfy $\big(u^*_{N, \rm eig}, u^*_{\rm eig}\big)_{L^2(\Omega)}\geq 0$, and for every $N \in \mathbb{N}$ let $\Pi^X_N \colon X \rightarrow X_N$ denote the $(\cdot, \cdot)_X$-orthogonal projection operator onto $X_N$. Then we have the error estimate
    \begin{align}\label{eq:Poly_eig}
      \Vert \Pi_N^X u^*_{\rm eig}-u^{*}_{N, \rm eig}\Vert_{L^2(\Omega)} +N^{-1}\Vert \Pi_N^X u^*_{\rm eig}-u^{*}_{N, \rm eig}\Vert_{H^1(\Omega)}&\lsim N^{-3}\Vert \Pi_N^X u^*_{\rm eig}-u^{*}_{\rm eig}\Vert_{H^{1}(\Omega)}.
    \end{align}
\end{theorem}

\begin{proof}
    Analogously to the proof of Theorem \ref{thm:2a} above, the proof of the present theorem is based on a verification of \textbf{Assumptions B.1, B.2,} and \textbf{B.3} which allow us to apply Lemma \ref{lem:combine_eig_2}. With the exception of Estimate \ref{eq:assum_4new} in \textbf{Assumption B.3}, these verifications are virtually identical to the ones carried out for \textbf{Assumptions A.1, A.2,} and \textbf{A.3} in the proof of Theorem \ref{thm:2a}. Additionally, the verification of Estimate \ref{eq:assum_4new} in \textbf{Assumption B.3} in the present spectral polynomial case is very similar to that carried out for the finite element case in the proof of Theorem \ref{thm:1b} with some obvious modifications. For the sake of brevity, we desist from repeating these arguments. 
    
    Let us nevertheless make the observation that in the present setting, the exact constrained minimiser $u^*_{\rm eig}\in X$ satisfies the equation
    \begin{align*}
        -\Delta u^*_{\rm eig} = \lambda^* u^*_{\rm eig} - V u^*_{\rm eig} + (u^*_{\rm eig})^3.
    \end{align*}
    Consequently, we can appeal to Lemma \ref{lem:appendix} in the appendix to deduce from elliptic regularity theory that $u^*_{\rm eig} \in H^4(\Omega)$ (compared to $u^*_{\rm src} \in H^{3-\epsilon}(\Omega)$ in Theorem \ref{thm:2a}). This allows us to prove the uniform boundedness of the discrete constrained minimisers $\{u^*_{N, \rm eig}\}_{N \in \mathbb{N}}$ in $H^2(\Omega)$, which, in turn, explains the $\epsilon$-improvement in the results of Theorem \ref{thm:2b} compared to Theorem \ref{thm:2a}.

\end{proof}

\section{Application to Spectral Fourier Discretisations}\label{sec:6}

Throughout this section, we assume the periodic framework of \text{Setting Two}. In particular, $\Omega$ is assumed to be a $d$-dimensional torus that is isomorphic to the unit cell $\tilde{\Omega}$ of a periodic lattice $\mathbb{L}$. For simplicity, we assume that $\widetilde{\Omega}= (0, 1)^d$ although other choices do not affect the result. For this choice of $\widetilde{\Omega}$, the orthonormal Fourier basis $\mathcal{B}$ of $L^2(\Omega)$ is given by
	\begin{align*}
		\mathcal{B} := \left\{e_{\bold{G}}(\bold{x}):=\frac{1}{\vert \Omega \vert^{\frac{1}{2}} }e^{\imath \bold{G} \cdot \bold{x}}\colon \quad  \bold{G}\in 2\pi \mathbb{Z}^d \right\}.
	\end{align*}
Next, we define the sequence of approximation spaces $\{X_{N}\}_{N \in \mathbb{N}}$ that we consider in this section~as
\begin{align}\label{eq:fourier_spaces}
\forall N \in \mathbb{N} \colon \quad X_{N}:= \text{span} \Big\{e_{\bold{G}}\colon ~ \bold{G} \in 2\pi \mathbb{Z}^d  ~\text{ with }~ \vert \bold{G} \vert \leq 2\pi N\Big\}.
\end{align}
Thus, similar to Section \ref{sec:5}, $N^{-1}$ will play the role of the discretisation parameter $\delta$ from Sections \ref{sec:source} and \ref{sec:eig}.

As in Section \ref{sec:5} -- and for similar reasons -- we will impose a stronger regularity assumption on the effective potential $V$ and source function $f$, namely, that $V \in H^{r_v}(\Omega)$ for some $r_v \geq2$ and $f \in H^s(\Omega)$ for some $s \geq 1$. With this additional assumption, the application of Lemmas~\ref{lem:combine_1} and \ref{lem:combine_2} to this choice of Galerkin approximation spaces results in the following theorem.

\begin{theorem}\label{thm:3a}
    Assume that \textbf{Setting Two} holds and $V \in H^{r_v}(\Omega)$ and $f \in H^s(\Omega)$ for some $r_v \geq2$ and $s \geq 1$ respectively, let $u^*_{\rm src}\in X$ denote the unique minimiser of the continuous energy functional \eqref{eq:1.1}, let the spectral Fourier approximation spaces $\{X_{N}\}_{N \in \mathbb{N}}$ be defined according to Equation \eqref{eq:fourier_spaces}, let $u^*_{N, \rm src}\in X_N, ~ N \in \mathbb{N}$ denote the unique discrete minimiser of the discrete energy functional \eqref{eq:1.5}, and for every $N \in \mathbb{N}$ let $\Pi_N \colon X \rightarrow X_N$ denote the $(\cdot, \cdot)_X$-orthogonal projection operator onto $X_N$. Then we have the error estimate
    \begin{align}\label{eq:Fourier_src}
      \Vert \Pi_N u^*_{\rm src}-u^{*}_{N, \rm src}\Vert_{L^2(\Omega)} +N^{-1}\Vert \Pi_N u^*_{\rm src}-u^{*}_{N, \rm src}\Vert_{H^1(\Omega)}&\lsim N^{-3}\Vert \Pi_N u^*_{\rm src}-u^{*}_{\rm src}\Vert_{H^{1}(\Omega)}.
    \end{align}
\end{theorem}
\begin{proof}
    We follow the template of Theorem \ref{thm:2a} and demonstrate the validity of \textbf{Assumptions A.1, A.2,} and \textbf{A.3} following which, we apply Lemmas \ref{lem:combine_1} and \ref{lem:combine_2}. Note that since the underlying function space $X=H^1(\Omega)$ is defined on a torus and we use spectral Fourier approximation spaces, the $(\cdot, \cdot)_X$-orthogonal projection operator $\Pi_N$ is also $H^s$-orthogonal for any $s\geq 0$ (which explains why we have dropped the superscript $X$ in our notation for this operator). This fact will be referenced repeatedly in the sequel and will considerably simplify the proof.

The verification of \textbf{Assumption A.1} is straight-forward. Indeed, making use of the fact that the governing equation is now posed on a torus, we can easily show that $u^*_{\rm src} \in H^{3}(\Omega)$. It is therefore a trivial exercise to show that the sought-after estimate \eqref{eq:assum_1a} holds for $r_{\rm src}= 3$. It is similarly clear that the orthogonal projection operator $\Pi_N\colon X \rightarrow X_N, ~ N \in \mathbb{N}$ satisfies the estimate \eqref{eq:assum_1b}. 

To verify \textbf{Assumption A.2}, we can use the same strategy used in the proof of Theorem \ref{thm:2a} and demonstrate the stronger result that
\begin{align}\label{eq:revision_hassan_new_verify_fourier}
    \exists C>0, ~\forall N\in \mathbb{N} \colon \qquad \Vert u^*_{N, \rm src} \Vert_{H^{2}(\Omega)} \leq C.
\end{align}
In order to do so, we simply replace the operator $ \widetilde{\Pi}^{2, 1, 0}_N$ used in the spectral polynomial case (Theorem \ref{thm:2a}) with the standard $(\cdot, \cdot)_X$ projection operator, and we make use of the following inverse inequality satisfied by the spectral polynomial approximation spaces $\{X_N\}_{N \in \mathbb{N}}$:
\begin{align*}
 \forall v_N \in X_N \colon \qquad   \Vert v_N\Vert_{H^{p}(\Omega)} \lsim N^{p-1}\Vert v_N\Vert_{H^1(\Omega)}.
\end{align*}
Since the arguments are otherwise identical to those used in the proof of Theorem \ref{thm:2a}, we refrain from presenting the details.

Turning now to the verification of \textbf{Assumption A.3}, we introduce, for each $N \in \mathbb{N}$, the auxiliary potential function $\widetilde{V}_N \colon \Omega \rightarrow \mathbb{R}$ given by (cf. Definition \ref{def:v_aux})
\begin{align}
    \widetilde{V}_N:= V + \left(u^*_{\rm src}\right)^2 + \left(u^N_{\rm src}\right)^2 + u^*_{\rm src}u^N_{\rm src} -1. 
\end{align}

Thanks to the $L^2(\Omega)$-orthogonality of the $(\cdot, \cdot)_X$-orthogonal projection operators $\Pi_N$, we deduce that for all $N$ sufficiently large it holds that

\begin{align}\nonumber
\int_\Omega \widetilde{V}_N({\bf x})  \varphi({\bf x}) [\chi- \Pi_N \chi]({\bf x})\; d{\bf{x}}&=\int_\Omega (I - \Pi_N) (\widetilde{V}_N  \varphi)({\bf x}) [\chi-\Pi_N \chi]({\bf x})\; d{\bf{x}}\\ \nonumber
&\lsim 
\|(I - \Pi_N) (\widetilde{V}_N   \varphi)\|_{L^2(\Omega)} \| \chi-\Pi_N \chi\|_{L^{2}(\Omega)}\\[0.25em] \label{er1c}
&\lsim N^{-1}\|(I - \Pi_N) (\widetilde{V}_N \varphi)\|_{L^2(\Omega)} \| u^*_{\rm src}-\Pi_N u^*_{\rm src}\|_{X}.
\end{align}
with the last step following from Estimate \eqref{eq:assum_1b} in \textbf{Assumption A.1}.

It remains to estimate the second term appearing on the right-hand side of Inequality \eqref{er1c} above. To do so, we first note that, by assumption, the effective potential $V \in H^{r_v}(\Omega)$ for some $r_v \geq 2$ so that Estimate \eqref{eq:revision_hassan_new_verify_fourier} implies that $\{\widetilde{V}_N\}_{N \in \mathbb{N}}$ is uniformly bounded in $H^{2}(\Omega)$. Thus, for all $N\in \mathbb{N}$ it holds that
\begin{align*}
     \|(I - \Pi_N) (\widetilde{V}_N  \varphi)\|_{L^2(\Omega)}  &\lsim  N^{-2}\| \widetilde{V}_N\|_{H^{2}(\Omega)} \|\varphi\|_{H^{2}(\Omega)},
\end{align*}
where we have used, as usual, the fact that $H^2(\Omega)$ is a Banach algebra. It therefore follows that \textbf{Assumption A.3} also holds with constant $q^*_{\rm src}= 2$, and we can thus apply Lemma \ref{lem:combine_2} with $q^*_{\rm src}=2$ to conclude.
\end{proof}

In a similar fashion, the application of Lemma \ref{lem:combine_eig_2} to the present choice of spectral Fourier Galerkin approximation spaces results in the following theorem.
\begin{theorem}\label{thm:3b}
    Assume that \textbf{Setting Two} holds and $V \in H^{r_v}(\Omega)$ for some $r_v \geq2$, let $u^*_{\rm eig}\in X$ denote the unique positive minimiser of the continuous energy functional \eqref{eq:1.2}, let the spectral Fourier approximation spaces $\{X_{N}\}_{N \in \mathbb{N}}$ be defined according to Equation \eqref{eq:fourier_spaces}, let $u^*_{N, \rm eig}\in X_N, ~ N \in \mathbb{N}$ denote a sequence of discrete minimisers of the discrete energy functional \eqref{eq:1.5} that satisfy $\big(u^*_{N, \rm eig}, u^*_{\rm eig}\big)_{L^2(\Omega)}\geq 0$, and for every $N \in \mathbb{N}$ let $\Pi_N \colon X \rightarrow X_N$ denote the $(\cdot, \cdot)_X$-orthogonal projection operator onto $X_N$. Then we have the error estimate
    \begin{align}\label{eq:Fourier_eig}
      \Vert \Pi_N u^*_{\rm eig}-u^{*}_{N, \rm eig}\Vert_{L^2(\Omega)} +N^{-1}\Vert \Pi_N u^*_{\rm eig}-u^{*}_{N, \rm eig}\Vert_{H^1(\Omega)}&\lsim N^{-3}\Vert \Pi_N u^*_{\rm eig}-u^{*}_{\rm eig}\Vert_{H^{1}(\Omega)}.
    \end{align}
\end{theorem}

\begin{proof}
   Analogously to the proof of Theorem \ref{thm:3a} above, the proof of the present theorem is based on a verification of \textbf{Assumptions B.1, B.2,} and \textbf{B.3} which allow us to apply Lemma \ref{lem:combine_eig_2}. With the exception of Estimate \ref{eq:assum_4new} in \textbf{Assumption B.3}, these verifications are virtually identical to the ones carried out for \textbf{Assumptions A.1, A.2,} and \textbf{A.3} in the proof of Theorem \ref{thm:2a} so we desist from repeating the arguments. The verification of Estimate \ref{eq:assum_4new} in \textbf{Assumption B.3} in the present spectral Fourier case is also straightforward. Indeed, since the $(\cdot, \cdot)_X$-orthogonal projector $\Pi_N$ onto the Fourier approximation space $X_N$ is also $(\cdot, \cdot)_{L^2(\Omega)}$, we have the following easily verifiable identity
   \begin{align}\nonumber
       (u^*_{\rm eig}- u^*_{N, \rm eig}, u^*_{\rm eig})_{L^2(\Omega)} &= (\Pi_N u^*_{\rm eig}- u^*_{N, \rm eig}, u^*_{\rm eig})_{L^2(\Omega)}+(u^*_{\rm eig}- \Pi_N u^*_{\rm eig}, u^*_{\rm eig})_{L^2(\Omega)} \\ \nonumber
       &= (\Pi_N u^*_{\rm eig}- u^*_{N, \rm eig}, u^*_{\rm eig})_{L^2(\Omega)}+(u^*_{\rm eig}- \Pi_N u^*_{\rm eig}, u^*_{\rm eig} - \Pi_N u^*_{\rm eig})_{L^2(\Omega)} \\ \nonumber
       &= \frac{1}{2} \Vert u^*_{N, \rm eig}- u^*_{\rm eig}\Vert^2_{L^2(\Omega)}- \frac{1}{2} \Vert \Pi_N u^*_{\rm eig}- u^*_{\rm eig}\Vert^2_{L^2(\Omega)}  +\Vert u^*_{\rm eig}- \Pi_N u^*_{\rm eig}\Vert^2 _{L^2(\Omega)}\\ 
       &= \frac{1}{2} \Vert u^*_{N, \rm eig}- u^*_{\rm eig}\Vert^2_{L^2(\Omega)} +\frac{1}{2} \Vert \Pi_N u^*_{\rm eig}- u^*_{\rm eig}\Vert^2_{L^2(\Omega)}.
       \label{eq:final_spectral}
   \end{align}
    Equation \eqref{eq:final_spectral} can thus be estimated using the higher-order convergence rates available for the $(\cdot, \cdot)_X$-orthogonal projector $\Pi_N$ and the spectral Fourier Galerkin approximation. We refer to \cite[Section 3]{cances2010numerical} for the details.
\end{proof}

\section*{Appendix}

The goal of this appendix is to prove the following simple result.

\begin{lemma}\label{lem:appendix}
    Let $\Omega = (-1, 1)^d$ with $d \in \{1, 2, 3\}$, let $s \in [1, 2]$, let $g \in H^{s}(\Omega) \cap H_0^1(\Omega) \cap \mathcal{C}^0(\overline{\Omega})$, and let $\Theta := (-1, 3)^d$. Then, there exists a periodic extension $g_{\rm per} \in H^s_{\rm per}(\Theta)$ of $g$.  Additionally, the unique solution $u \in H_0^1(\Omega)$ to the homogeneous Poisson equation
    \begin{align}\label{eq:appendix_dir}
        -\Delta u = g \qquad \hspace{3mm}\text{in } H_0^1(\Omega),  
    \end{align}
    coincides on $\Omega \subset \Theta$ with the unique solution $u \in H^{1}_{\rm per}(\Theta)$ of the periodic Poisson equation
     \begin{align}\label{eq:appendix_per}
        -\Delta u_{\rm per} = g_{\rm per} \qquad \text{in } H^1_{\rm per}(\Theta).
    \end{align}
\end{lemma}
\begin{proof}
  We consider the case $s=2$ and $d\in \{2, 3\}$. The cases $d=1, s=2$ and  $d\in \{1, 2, 3\}, s=1$ are similar, and the remaining cases $d\in \{1, 2, 3\}, s \in (1, 2)$ follow by interpolation. 
  
  To begin with, recall that for $\Theta = (-1, 3)^d$, the Hilbert space $H^2_{\rm per}(\Theta)$ is defined as
  \begin{align*}
      H^2_{\rm per}(\Theta):= \Big\{u\in H^2_{\rm loc}(\mathbb{R}^d)\colon\quad  u(\bold{x} +\bold{G})=u(\bold{x}) \quad \text{for a.e. } \bold{x}\in \mathbb{R}^d, ~\forall \bold{G}\in 4\mathbb{Z}^d\Big\},
  \end{align*}
equipped with the usual $H^2(\Theta)$ inner product.

Next, observe that we can extend the domain of definition of~$g$ from $(-1, 1)^d$ to $\Sigma:=(-1, 3) \times (-1, 1)^{d-1}$ by antisymmetry, i.e., by setting 
\begin{align}\label{eq:appendix_1}
h(x, \bold{y})= \begin{cases}
          -g(2-x, \bold{y}) \qquad &\text{if } (x,\bold{y}) \in (1, 3) \times (-1, 1)^{d-1}\\
          g(x, \bold{y}) \qquad &\text{if } (x, \bold{y}) \in (-1, 1)\times (-1, 1)^{d-1}
     \end{cases}
 \end{align}

We claim that the extended function $h \in H^2(\Sigma)$. To see this, let us denote $\widetilde{\Omega}= \Sigma\setminus \overline{\Omega}= (1, 3)\times (-1, 1)^{d-1}$ and introduce the restriction $\widetilde{g}= h\vert_{\widetilde{\Omega}}$. Since the original function $g\in H^2(\Omega)$ by assumption, we must have that $\widetilde{g}\in H^2(\widetilde{\Omega})$. The construction of $h$ also implies that the trace and normal trace of $\widetilde{g}$ coincide with those of $g$ on the mutual boundary $\partial \widetilde{\Omega}\cap \partial \Omega= \{1\}\times (-1, 1)^d$. It is therefore a consequence of Green's identity that $h\in H^2(\overline{\widetilde{\Omega} \cup \Omega})=H^2(\Sigma)$. Moreover, since $h \in \mathcal{C}^0(\overline{\Sigma})$, we also have by construction that  $h \in H_0^1(\Sigma)$.

Repeating the same process first in the other dimensions yields an extended function $h \in H_0^1(\Theta) \cap H^2(\Theta) \cap \mathcal{C}^0(\overline{\Theta})$. Consequently, we may apply the same basic construction iteratively to obtain an extension $g_{\rm per} \in H^2_{\rm loc}(\mathbb{R}^d)$, and since we extend our function each time by anti-symmetry (see Equation \eqref{eq:appendix_1}), we have that $g_{\rm per}$ is $4\mathbb{Z}$-periodic, i.e., $g_{\rm per} \in H^2_{\rm per}(\Theta) \subset H^2_{\rm loc}(\mathbb{R}^d)$.

It remains to verify that the solution $u \in H_0^1(\Omega)$ to the homogeneous Dirichlet problem \eqref{eq:appendix_dir} coincides on $\Omega$ with the solution $u_{\rm per}\in H^1_{\rm per}(\Theta)$ to the periodic problem \eqref{eq:appendix_per}. Since $u \in H_0^1(\Omega)$ is the unique solution to Equation \eqref{eq:appendix_dir} and $-\Delta u_{\rm per}= g_{\rm per}=g$ on $\Omega$, it suffices to verify that $u_{\rm per}$ has trace zero on $\partial \Omega$. This, in fact, follows from the observation that the the periodic extension $g_{\rm per}$ is, by construction, symmetric-odd with respect to each corner of $\Omega= (-1, 1)^d$. Thus, $u_{\rm per}$ must also be symmetric-odd with respect to each corner of $\Omega$ and consequently of trace zero $\partial \Omega$ as required.
\end{proof}

Consider the setting of Lemma \ref{lem:appendix} above. It is important to note that since the periodic extension $g_{\rm per}\in H^2{\rm per}(\Theta)$ for $s \in [1, 2]$, elliptic regularity theory in the periodic setting implies that $u_{\rm per}\in H_{\rm per}^{2+s}(\Theta)$. And since $u_{\rm per}$ coincides on $\Omega$ with $u\in H^1_0(\Omega)$ that solves the homogeneous Dirichlet problem \eqref{eq:appendix_dir}, we must have $u \in H^{2+s}(\Omega)$. This is despite the fact that Equation \eqref{eq:appendix_dir} is posed on a polyhedron where elliptic regularity theory yields, in general, that $u$ belongs at most to $H^{t}(\Omega)$ for any $t < 3$. This trick is used several times in the present work.

\section*{Acknowledgements}

This work has been funded by the European Research Council (ERC) under the European Union’s Horizon 2020 research and innovation program (Grant No. 810367), project EMC2 (Y.M.).

\bibliographystyle{abbrv.bst}
	\bibliography{refs.bib}

\end{document}